\documentclass[12pt, oneside]{article}  
\pdfoutput=1 	
\usepackage[margin=.75in]{geometry}      
\usepackage{enumitem}
\usepackage{graphicx}				
\usepackage{amsmath}
\usepackage{amssymb}
\usepackage{amsthm}
\usepackage{cleveref}
\usepackage[margin=.6in, font=small]{caption}
\usepackage{fancyhdr}
\newtheorem{theorem}{Theorem}[section]
\newtheorem{lemma}[theorem]{Lemma}

\newtheorem{proposition}[theorem]{Proposition}

\theoremstyle{definition}
\newtheorem{definition}[theorem]{Definition}
\newtheorem{example}[theorem]{Example}
\newtheorem{remark}{Remark}

\def\D{{\em Dot}}

\newcommand\Dec[1]{ Dec_{#1}(S) }

\usepackage{marginnote,cprotect}
\usepackage[obeyFinal]{todonotes}
\usepackage[color]{changebar}
\usepackage{soul}                     
\usepackage{ifdraft}
\definecolor{revcol1}{rgb}{0.0078,0.2980,0.7961}
\definecolor{revcol2}{rgb}{0.6118,0.1765,1.0000}
\definecolor{revcol3}{rgb}{1.0000,0.6431,0.2627}

\date{\today}				

\title{Distance and intersection number in the curve graph of a surface}
\author{Joan Birman, Matthew J. Morse, and Nancy C. Wrinkle}

\begin{document}
\maketitle
\begin{abstract}

In this work, we study the cellular decomposition of $S$ induced by a filling pair of curves $v$ and $w$, $Dec_{v,w}(S) = S \setminus (v \cup w)$, and its connection to the distance function $d(v,w)$ in the curve graph of a closed orientable surface $S$ of genus $g$.
 Efficient geodesics were introduced by the first author in joint work with Margalit and Menasco in 2016, giving an algorithm that begins with a pair of non-separating filling curves that determine vertices $(v,w)$ in the curve graph of a closed orientable surface $S$ and computing from them a finite set of {\it efficient} geodesics.  
 We extend the tools of efficient geodesics to study the relationship between distance $d(v,w)$, intersection number $i(v,w)$, and $Dec_{v,w}(S)$.  The main result is the development and analysis of particular configurations of rectangles in $Dec_{v,w}(S)$ called \textit{spirals}.  We are able to show that, with appropriate restrictions, the efficient geodesic algorithm can be used to build an algorithm that reduces $i(v,w)$ while preserving $d(v,w)$.  At the end of the paper, we note a connection between our work and the notion of extending geodesics.
\end{abstract}

\section{Introduction} \label{s:introduction}  
The {\it curve graph} $\mathcal C(S)$ of a closed orientable surface $S=S_g$ of genus $g \geq 2$ is the metric graph whose vertices correspond to isotopy classes of essential simple closed curves in $S$.
Edges join vertices that have disjoint representatives in $S$.
Each edge is defined to have length 1. 
Let $v,w$ be vertices in $\mathcal{C}(S)$.
The {\it distance} $d(v,w)$ is the length of a shortest path in the curve graph from $v$ to $w$.
Any shortest path is a {\it geodesic} from $v$ to $w$ in $\mathcal C(S)$. We often denote a geodesic between two vertices $v$ and $w$ by listing the intermediate curves comprising the vertices in the path connecting them: $v=v_0, v_1, v_2, \ldots, v_{d-1}, v_d=w$. 

In 2002, the PhD thesis of Jason Leasure \cite{Lea} presented an algorithm to compute the exact distance between two vertices  of $C(S)$, noting that there was no hope of implementing his algorithm for concrete computations.
Fourteen years later, Birman, Margalit and Menasco \cite{BMM} produced a much faster algorithm in a spirit similar to that of \cite{Lea}. The same year, the efficient geodesic algorithm of \cite{BMM} was partially implemented in the computer program MICC (Metric in the Curve Complex). A wealth of low-distance examples provided by the program has motivated much of our work. The paper \cite{GMMM} explains the implementation of the algorithm and the results derived from it; the program itself and its documentation is available for free download at \cite{MICC}. Our work in this paper was inspired by a desire to use the data provided by \cite{MICC} to begin to create a bridge between the newly available tools in \cite{BMM} and the mainstream work of the past 20 years on the large-scale geometry of the curve graph (starting with \cite{MM-I}, \cite{MM-II}.) 

Here is a guide to this paper and a description of its content.   The material in $\S$\ref{s:background} is background for the rest of the paper, consisting of results that are either known or close to known.  
In $\S$\ref{ss:decomposition}, we study the decomposition of the surface $S$ we obtain by cutting it open along a pair of filling curves into polygons. We derive several interesting and useful equations from Euler characteristic equations, which guide us on how to approach our work. In $\S$\ref{ss:efficient}, we review the material from \cite{BMM} that we will need later, including efficient geodesics and surgery.  After that, we will be ready to begin our new work. In $\S$\ref{s:spiral surgery}, we describe how to extend the surgeries reviewed in $\S$\ref{ss:efficient} to the endpoints of a geodesic.  In $\S$\ref{subsection:motivating eg}, we do a simple motivating example. In $\S$\ref{subsection:spirals}, we define {\em spirals} and explore their properties. In $\S$\ref{subsection:spiral surgery}, we describe new geometric realizations of the surgeries from $\S$\ref{ss:efficient}, which we call {\em spiral surgery}. In $\S$\ref{s:reduce-i}, we consider  how to recognize when we can perform surgery on the endpoints of a path in the curve graph, that is, when it will preserve adjacency of vertices in the curve graph and when it will preserve distance.  In $\S$\ref{ss:extended-dot-graphs}, we extend the notion of the dot graph reviewed in $\S$\ref{ss:efficient} to give criteria for when surgery will maintain adjacency in the curve graph, and in $\S$\ref{ss:main-theorem}
we apply the notion of the stacked, extended dot graph to the setting of a spiral to prove in Theorem~\ref{theorem:main} that the conditions in $\S$\ref{ss:extended-dot-graphs} describe the criteria under which spiral surgery preserves both efficiency and distance. In our conclusion, $\S$\ref{sec:conclusion}, we explore future directions implied by this work. Finally, we have included an Appendix with an extended example of spiral surgery on a familiar example of Hempel \cite{SS}, reducing its intersection number to the minimum for its decomposition. The details of that example point to several interesting future directions of work. \\

\noindent {\bf Acknowledgements:} We are grateful for an insightful counterexample offered to us by the referee for an earlier version of this work, which greatly improved the main theorem of this paper, as well as the careful reading they gave the manuscript. We are also grateful for helpful, clarifying conversations with Bill Menasco, Alex Rasmussen, Nick Salter, and Sam Taylor.  
The first author thanks the Simons Foundation for partial support, 
during the initial phases of this work, from a Collaborative Research 
Grant,  Award 245711. 
The third author thanks the initiative ``A Room of One's Own" for financial support and focused time to write. 
\section{Background} \label{s:background} We begin by examining the decomposition of a closed orientable surface obtained by cutting it open along two essential curves, then connecting the decomposition to the study of intersection number and distance in $C(S)$. Then we dive into some necessary background with a review of the tools and results from \cite{BMM} that we will use throughout the paper.

Throughout this paper, we assume that vertices in the curve graph are isotopy classes of  {\it non-separating} curves on the surface $S$ (see \cite{BM} for the underlying reasons for this assumption, and see \cite{Ras} for a review of properties of the non-separating curve graph.) We assume further that paths  in the curve graph $\mathcal C(S)$ that join vertices are always chosen to be efficient in the sense of \cite{BMM} (see Section \ref{ss:efficient} for the definitions of efficient paths.) 
We will follow the conventions of the literature in abusing notation to denote vertices in the curve graph $\mathcal C(S)$ as well as representative curves in each isotopy class by $v, w$. It will be clear from context when $v$ is a curve or an isotopy class of curves. 
We assume that each component of $S\setminus v \cup w$ is a topological disc; such a $v$ and $w$ are a \textit{filling pair} and are said to \textit{fill} $S$.
This condition implies that $d(v,w) \geq 3$. 
In addition, we assume that $v$ and $w$ are in \textit{minimal position}: $|v \cap w|$ is minimized with respect to the isotopy classes of $v$ and $w$. 
This implies that there are no bigons in $v\cup w$.
We then define the intersection number of $v$ and $w$, $i(v,w) = | v \cap w|$.

\subsection{The decomposition $\Dec{v,w}$} 
\label{ss:decomposition}

Much of the work of this paper is based on studying $S \setminus (v \cup w)$, the result of cutting $S$ open along $v$ and $w$ , which we refer to as the \textit{decomposition of $S$ by $v \cup w$} and denote by $\Dec{v,w}$. 
Since we assume the curves fill $S$, we know that $S \setminus (v \cup w)$ is a union of discs.
 Each disc is isotopic to a polygon with alternating edges in $v$ and $w$. We observe that each polygon has an even number of total sides. $\Dec{v,w}$ consists of more than a simple list of polygons of various sizes, however. $\Dec{v,w}$ should also keep track of how the cut-open curves are identified to reassemble the surface, so along with the polygons themselves we include the polygonal edge identifications, which means two copies of labels $1, 2, \ldots, i(v,w)$, one for each copy of $v$ and $w$. Throughout this paper, we use green labels for $v$ and red labels for $w$. 
 We note that this identification is also guided by the fact that our surface is orientable: we are gluing polygons together so that the ``positive" sides of the surface face out. 
Note that, while vertices in the curve graph are not oriented curves, our labeling requires an arbitrary choice of orientation. 
This choice does not affect any subsequent results.

\begin{figure}[ht]
  \centering
  \includegraphics[width=6in]{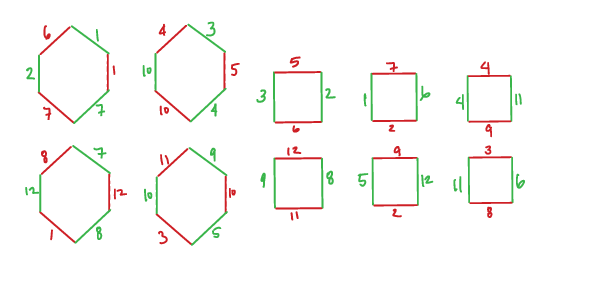}
\caption{A polygonal decomposition $\Dec{v,w} = S_g-\{v,w\}$. Here $g = 2, i=12,$ and $d=4$ (distance verified by MICC,) and the polygons have their positive side facing the reader. To see the same decomposition glued together along its red $w$-edges, see Figure \ref{fig:d3-to-d4}-right.}
  \label{fig:d4-polygons-only} 
\end{figure}

 It is useful and straightforward to understand the number and types of polygons that appear in $\Dec{v,w}$ for general $v, w,$ and $S_g$ under the assumptions above. Let $F_{2k}$ be the number of polygons in $\Dec{v,w}$ with $2k$ sides, for $k \geq 2 $ (since there are no bigons).
 We will refer to such polygons as \textit{$2k$-gons}.
 We call the vector of the numbers of each size of polygon for $k > 2$ the {\it decomposition vector} $\mathcal{F}=\{F_6, F_8, \ldots, F_{2k}, \ldots \}$. We will derive below that the length of this vector is, in fact, $4g-4$. Two decompositions are {\em equivalent} if they have the same decomposition vector and the labels on $v$ and/or $w$ differ by a cyclic permutation.  
 
We first make some general observations about $\Dec{v,w}$. Recall the well-known upper bound for the distance $d(v,w)$ that is attributed to Lickorish, but was first stated by Hempel in \cite{He}, relating distance and intersection number:

\begin{equation}\label{eq:Hempel upper bound}
d(v,w) \leq 2 \  {\rm log}_2(i(v,w)) + 2
\end{equation}  
This inequality can be interpreted in terms of the number of polygons in $\Dec{v,w}$: see Equation \ref{eq: inessential rectangles}, which we now explain.  

Let $V$ and $E$ be the number of vertices and edges in the graph induced by $v \cup w$ on $S$, with $F = |\Dec{v,w}|$
(Note that the vertices and edges in our discussion here about $\Dec{v,w}$ are not vertices in the curve graph.) 
Since every vertex has valence four and every edge touches two vertices, we have  $4V=2E$  or $2V=E$, so that  $-2\chi(S_g) = 4g-4 = 2(-V+E-F) = E-2F = 2V - 2F$. Therefore:
\begin{equation}\label{eq:Eulerchar-2}
F = F_4 +F_6 + F_8 + F_{10}+F_{12} +\cdots
\end{equation}
But now observe that each edge is an edge of exactly 2 polygons, so  that $2E =  4F_4 + 6 F_6 +8F_8 + \cdots$ or 
\begin{equation}
\label{eq:Eulerchar-3}
E =  2V = 2F_4 + 3 F_6 +4F_8 +5F_{10}+6F_{12} + \cdots
\end{equation}
where each term on the right and left of equation (\ref{eq:Eulerchar-3}) is non-negative.  

These Euler characteristic considerations show that:
\begin{equation}\label{eq:Euler characteristic}
4g-4 =  F_6 +2 F_8 +3 F_{10} + \cdots (4g-4)F_{8g-4}
\end{equation}
Equation (\ref{eq:Euler characteristic}) tells us that there are finitely many possibilities for the decomposition of $S$ into polygons larger than rectangles. 
Since every term on both sides of (\ref{eq:Euler characteristic}) is non-negative, we see that there are finitely many different decomposition vectors, and the largest possible polygon in $\Dec{v,w}$ has $8g-4$ sides.  

Note that the number $F_4$ of rectangles  does not appear in (\ref{eq:Euler characteristic}). Combining equations (\ref{eq:Eulerchar-2}) and (\ref{eq:Eulerchar-3}), and noting that $V/4= i(v,w)$,  we obtain the following equation:  

\begin{equation}\label{eq:intersection no and face decomposition}
i (v,w) = F_4 + \mathcal F_g, \ {\rm where} \ \mathcal F_g = (1/2)(3F_6+4F_8+\cdots + (4g-2)F_{8g-4})
\end{equation}
When combined with the inequality in Equation (\ref{eq:Hempel upper bound}), we get our first connection between distance and decomposition: 
\begin{equation}\label{eq: inessential rectangles}
d(v,w) \leq 2 \  {\rm log}_2(F_4 + \mathcal F_g) + 2.    
\end{equation}

Since we already learned that the largest possible polygon in $\Dec{v,w}$ has $8g-4$ sides and we have finitely many possibilities for $\mathcal F_g$, the parameter $F_4$ in (\ref{eq: inessential rectangles}) plays the role of $i(v,w)$ in (\ref{eq:Hempel upper bound}) as $i$ gets large; that is, $i(v,w) \to \infty$ implies that $F_4 \to \infty$.
Given this relationship between intersection number and rectangles, our primary effort in this paper is to provide an algorithm that, given a pair of vertices $(v,w)$, produces a new pair $(v',w')$ with lower intersection number by reducing $F_4$, while preserving both $\mathcal F_g$ and the distance $d(v,w)$. A question arises naturally here: is reducing the number of rectangles the only way to reduce intersection number while preserving distance? Surely not. But in this current project, we have applied the tools from \cite{BMM} to rectangles as a simple first effort, since the observation that $F_4 \to \infty$ as $i(v,w) \to \infty$ gave an obvious initial direction for our work.

In the context of (\ref{eq: inessential rectangles}), observe that the problem of minimizing $i(v,w)$ or $F_4$, while preserving $\mathcal F_g$ and $d(v,w)$, goes hand in hand with the problem of increasing $F_4$  and increasing $d(v,w)$ while preserving $\mathcal F_g$.
This is because one must understand how to recognize when $d(v,w)$ decreases as we reduce $F_4$. 
Thus, we will be implicitly studying the growth of $d(v,w)$ as $F_4$ increases as we study the reduction of $F_4$ and preservation of $d(v,w)$.  We discuss this connection further in $\S$\ref {sec:conclusion}. 

\begin{remark}
It is standard to define the minimum intersection number possible for $v$ and $w$ over all choices of vertices $v,w \in \mathcal{C}(S)$ having distance $d$ on the surface $S_g$. That is, the minimum intersection number depends only on distance and genus.  
However, the data told us that, at low distances, $\Dec{v,w}$ influences the minimum intersection number. The data produced by MICC in \cite{GMMM} was mainly for the case $g=2$ and  $d=3$ and 4, with incomplete data for $g=3, d=4$. Although the data we studied from MICC are very specific ($g=2$, $d(v,w)\geq4$), we gained significant insight.  Upon further inspection of these examples, we discovered that, for small $i(v,w)$, the possible face decompositions depend strongly on intersection number.  With that realization, we denote by $i_{min}(d, g, \mathcal{F})$ the minimum intersection number of a non-separating filling curve pair $(v,w)$ with fixed distance $d$, genus $g$, and decomposition vector $\mathcal{F}$. 
\end{remark}
The important property of the data set that we studied is that it contained all isotopy classes of distance $\geq 4$ with $i(v,w) \leq 25$ and genus $g = 2$.
The reasonable size of the examples enabled rapid experimentation and verification of calculations, and the completeness of the data allowed us to make conclusions about $d(v,w)$ with respect to $\Dec{v,w}$.
The key observation from the data that is relevant to this work is that the minimum intersection number $i_{min}(d, g, \mathcal{F})=i_{min}(4,2, \mathcal{F})=12$ is not always realizable for a given $\mathcal{F}$. To be precise, if $ \mathcal{F} = (F_6,F_8,F_{10},F_{12})$ = (4,0,0,0), (2,1,0,0) or (0,2,0,0), then $i_{min}(4, 2, \mathcal{F}) = 12$.
However,  if $\mathcal{F} =(F_6,F_8,F_{10},F_{12})$ = (0,0,0,1) or (1,0,1,0), then $i_{min}(4, 2, \mathcal{F})>12$. 

We still see that for distances 3 and 4, $i_{min}$ depends not only on genus and distance, but also on whether the decomposition in (\ref{eq:Euler characteristic}) is, at one extreme, into a family of $(4g-4)$ hexagons  or, at the other extreme, into a single $(8g-4)$-gon. For distance $3$ and genus $2$, we see that the minimum intersection number, $i=4$, occurs when the decomposition includes a single $12$-gon, while we found an $i=10$ example for a decomposition into 4 $6$-gons (Figure \ref{fig:d3-to-d4}-left,) which we don't know is minimal for that decomposition. 

It is known from the work of Aougab and Taylor \cite{AT} that at `large'  distances, the minimum intersection number $i_{min}(d,g)$ is actually independent of genus. We do not yet see hyperbolic behavior emerging at this low distance and granular level, however, so genus still plays a critical role in the $i_{min}$ considerations. 


\subsection{Background on efficient geodesics and surgery}
\label{ss:efficient} 

The work in this paper is entirely in the realm of {\em efficient} geodesics in the curve graph, which were defined and developed in \cite{BMM}. Before we define efficiency for a geodesic, we quote the following result from \cite{BMM} to motivate their use:

\begin{theorem}
{\bf (Theorem 1.1 of \cite{BMM}):} Let $g \geq 2$. If $v$ and $w$ are vertices of $\mathcal{C}(S_g)$ with $d(v, w) \geq 3$, then there exists an efficient geodesic from $v$ to $w$. What is more, there is an explicitly computable list of at most $n^{6g-6}$ vertices $v_1$ that can appear as the first vertex of an initially efficient geodesic $v = v_0,v_1, \ldots,v_n = w$. In particular, there are finitely many efficient geodesics from $v$ to $w$.
\label{thm:bmm-main}
\end{theorem}


An efficient geodesic from $v$ to $w$ satisfies certain upper bounds on, in a generalized sense, the intersection of each intermediate curve $v_k$ with arcs of $w$. We make precise these conditions in the definitions below. 

\begin{definition} [Reference arc] Let $v_0,\ldots , v_n$ be a geodesic of length at least three in $\mathcal{C}(S)$, and let $v_0, v_1$, and $v_n$ be pairwise in minimal position (this configuration is unique up to isotopy of $S$.) A {\em reference arc} for the triple $v_0, v_1, v_n$ is an arc $\alpha$ that is in minimal position with $v_1$ and whose interior is disjoint from $v_0 \cup v_n$; such arcs were first considered by Leasure [\cite{Lea}, Definition 3.2.1]. See Figure \ref{fig:reference_arcs} for examples. 
\end{definition}

We begin with bounding the intersections of $v_1$ with all possible reference arcs for $v, v_1, w$: 
\begin{definition} [Initially efficient geodesic] Choose $(v,w)$ with $d = d(v,w)\geq 3$.  Let
$\mathcal{G}:\, v=v_0, v_1, \cdots, v_{d-1}, v_d=w$ be a geodesic from $v$ to $w$. Then $\mathcal  G$ is {\bf initially efficient} if $|v_1 \cap \alpha| \leq d-1$ for all $\alpha$, where $\alpha$ is a reference arc for $v, v_1, w$.
  \end{definition} 
  
  Then, if each subgeodesic of the geodesic from $v$ to $w$, meaning from $v_1$ to $w$, from $v_2$ to $w$, and so forth is initially efficient, then we have an efficient geodesic: 
 
 \begin{definition}[Efficient]
A geodesic $\mathcal G$ is {\bf efficient} if the oriented sub-geodesic $v_k,\dots,v_d$ is initially efficient for each $0\leq k\leq d-3$ and the oriented geodesic $v_d, v_{d-1}, v_{d-2}, v_{d-3}$ is also initially efficient.   
\end{definition}

Reference arcs were developed in \cite{BMM} without the context of the polygonal decomposition $\Dec{v,w}$. We realized, however, that one gains a natural intuition for reference arcs and efficiency when considering them in the context of $\Dec{v,w}$. See Figure \ref{fig:reference_arcs} for all possible reference arcs relative to $v, v_1,$ and $w$ for some $2k$-gons. (We note that the reference arcs are parallel to arcs of $w$ in $4$-gons and $6$-gons.) One can see from the figure that initial efficiency, minimizing the intersections of $v_1$ with these reference arcs, means minimizing the number of times $v_1$ travels through a polygon, across all of the different ways $v_1$ can travel through that polygon, across all polygons in the decomposition. To reach efficiency from initial efficiency, we move through the geodesic, bounding at each stage the intersection of $v_k$ with all of the reference arcs relative to $v_{k-1}, v_k,$ and $w$, so we are strictly limiting the behavior of the $v_k$. The surprising accomplishment of the tools created in \cite{BMM} is that one can understand and manage these intersections at each point of an efficient geodesic simultaneously. 

\begin{figure}[ht]
  \centering
  \includegraphics[width=6in]{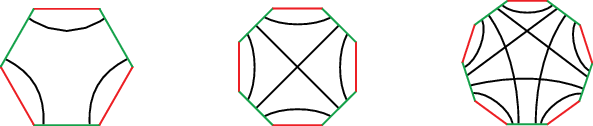}
  \caption{Examples of all possible reference arcs (relative to $v, v_1,$ and $w$,) on $2k$-gons when $k=3,4,5$.  The reference arcs are colored black, $v$ is colored green, and $w$ is red.} 
  \label{fig:reference_arcs}
  \end{figure}
  
Given any geodesic from $v$ to $w$ in the curve graph, \cite{BMM} showed how to make that geodesic efficient by means of surgery, as defined below. The surgery is performed simultaneously on collections of intermediate curves in the geodesic, preserving the adjacency and non-adjacency of those curves in the curve graph, while leaving $v$ and $w$ fixed. 
After a sufficient number of surgeries, the intersection criterion of the efficient geodesic definition is satisfied. 
We review their construction here. 
 
\begin{definition}[Surgery] \label{definition:general-surgery} Given a simple closed curve $v_k$ and an arc $\alpha$ whose endpoints lie on $v_k$ (and whose intersection with $v_k$ is only those endpoints),  {\bf surgery on $v_k$ with surgery arc $\alpha$} consists of gluing one of the two components of $v_k - \partial \alpha$ to $\alpha$ and discarding the other component of $v_k - \partial \alpha$. The choice of component to glue to $\alpha$ is made such that the resulting curve is closed. 
\end{definition}

If the curves were in minimal position before surgery, and if $v_k$ was essential, then it will be essential after surgery, thanks to a variant of the bigon criterion (see, for example, \cite{FLP}, Proposition 3.10.)  

Surgeries in \cite{BMM} were distinguished by whether the joined arcs are above (+) or below (-) the surgery arc $\alpha$, as shown in Figure \ref{fig:surgery}. We will categorize surgeries the same way. The $+$ and $-$ labels in the types of surgery are independent of the orientation of the curves. 
\begin{figure}[ht]
  \centering
  \includegraphics[width=6in]{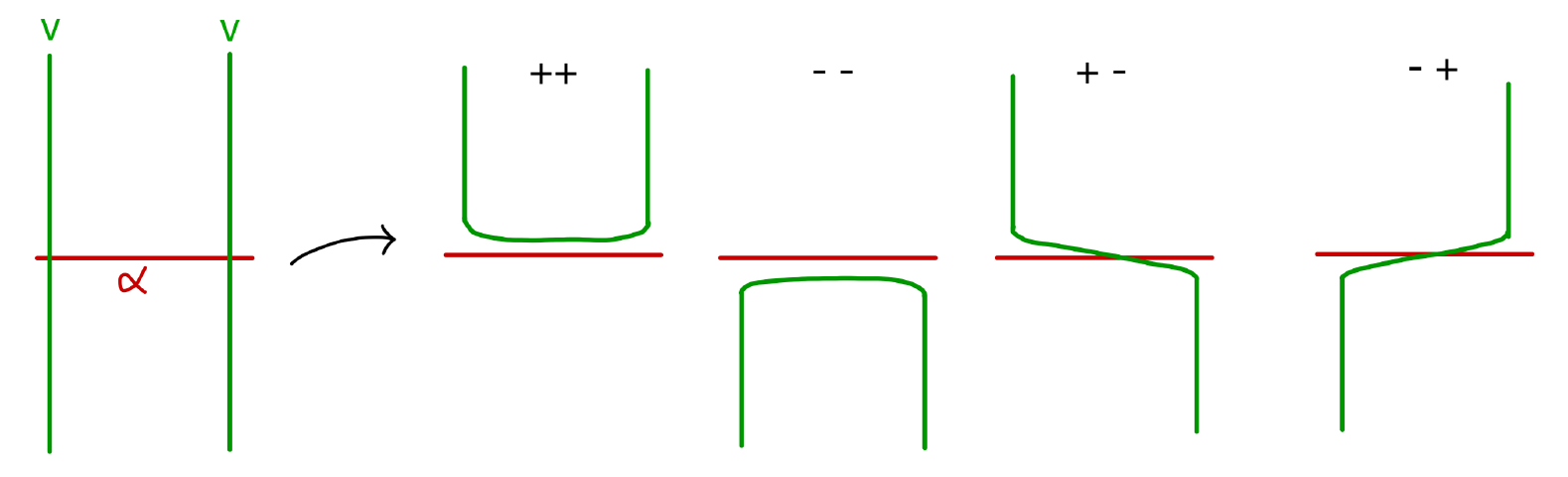}
  \caption{The four different possible surgeries on a curve $v$ with surgery arc $\alpha$, as defined in \cite{BMM}. The surgeries are denoted by whether the arcs joined after surgery are above $(+)$ or below $(-)$ the surgery arc. The unjoined part of the curve is deleted. One observes that surgery reduces $i(\alpha, v)$ by $1$ or $2$.}
  \label{fig:surgery}
\end{figure}

The main challenge in performing simultaneous surgery on a collection of curves is that curves that were disjoint before surgery should stay disjoint. For then, the collection of surgeries sends paths in the curve graph to other paths. Note here that \cite{BMM}'s surgeries were on the intermediate curves in a path, leaving endpoints $v$ and $w$ fixed, so the distance was always unchanged.  (By contrast, we will be extending these surgeries to the endpoints, $v$ and $w$.) \\

As an example, see Figure \ref{fig:sevenfold-surgery} (copied from \cite{BMM}.) On the left is a collection of curves, $v_2, v_3, \ldots, v_7$, labeled by each curve's index, that are part of some path from $v$ to $w$ in the curve graph. They are crossing some reference arc for the triple $v, v_1, w$.  The numbers at the top of the figure give an {\em intersection sequence} (defined below) for this collection with respect to this reference arc. On the right we see the result of a collection of simultaneous surgeries on these curves, along with a new sequence of numbers. These surgeries do not change the curves outside of the pictured area, so they preserve adjacency of the curves in the curve graph. That is, the only intersections after surgery occur between curves that already intersected. How did \cite{BMM} recognize when such collections of surgeries were possible, how did they know which surgeries to choose for each curve, and how did they know when no further surgeries were possible? The answer is in a visual representation, called a dot graph, of the intersections of the curves with the surgery arc. We now review that definition. 

\begin{figure}[ht]
  \centering
  \includegraphics[width=6in]{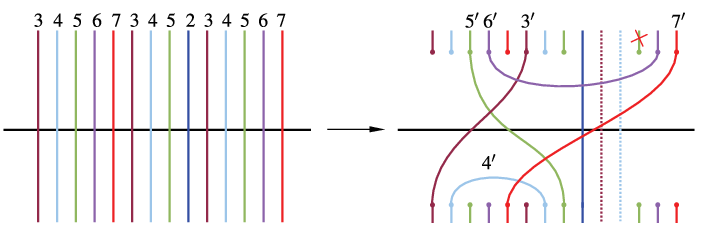}
  \caption{On the left is a collection of curves, $v_2, v_3, \ldots, v_7$ in some path, labeled by each curve's index, crossing some reference arc. The numbers at the top of the figure give an {\em intersection sequence} for this collection with respect to this reference arc. On the right we see a collection of simultaneous surgeries. These surgeries do not change the curves outside of the pictured area, so they preserve adjacency of the curves in the curve graph. That is, the only intersections after surgery occur between curves that already intersected.}
  \label{fig:sevenfold-surgery}
\end{figure}

\begin{definition} [Intersection sequence] Let $\mathcal G$ be an efficient path between $v$ and $w$, and $\alpha$ a reference arc for $v, v_1,$ and $w$. Following the construction in \cite{BMM}, we let $v_1, v_2, \ldots,v_{d-1}$ be vertices of $\mathcal G$, and let $N$ denote the cardinality of $\alpha \cap (v_1 \cup v_2 \cup \ldots \cup v_{d-2})$. Traversing $\alpha$ in the direction of some chosen orientation, we record the sequence of natural numbers $\sigma = ( j_1, j_2, \ldots , j_N) \in \{1, \ldots, d-2 \}^N$ such that the $i^{th}$ intersection point of $\alpha$ with $v_1 \cup v_2 \cup \ldots \cup v_{d-2}$ lies in $v_{j_i}$. Then $\sigma$ is called the {\bf intersection sequence}  of $\alpha$ of $\{v_i\}$.
\end{definition} 

In Figure \ref{fig:sevenfold-surgery}, for example, the intersection sequence for the pre-surgery configuration is written above the arcs of the curves $v_2, v_3, \ldots, v_7$ intersecting $\alpha$. 

It is helpful to simplify the intersection sequence so that we can  recognize more easily when surgery is possible. To that end, given an intersection sequence $\sigma$ of a reference arc $\alpha$, we  put $\sigma$ into a special form, {\em sawtooth form}, via repeated permutations of adjacent indices. \cite{BMM} showed that putting an intersection sequence into sawtooth form preserves the adjacency of curves in the path, as long as the permutations are of entries representing curves that already intersect (i.e., indices with difference greater than 1,) since permuting adjacent indices in the intersection sequence means moving one curve past the other on a surface. 

 \begin{definition} [Sawtooth form] A sequence of natural numbers is in sawtooth form if it satisfies the condition that $j_{i+1}> j_i \Rightarrow j_{i+1} = j_i + 1$ for each $j_i$. 
\end{definition}

 Any intersection sequence (in fact, any sequence of natural numbers,) can be put into sawtooth form by repeatedly permuting far-apart pairs of numbers that are adjacent in the sequence \cite{BMM}. Once it is in sawtooth form, an intersection sequence can be graphed in a standardized way, as a collection of consecutive dots connected by line segments of slope 1: 

\begin{definition} [Dot graph] A {\bf dot graph} is a visual representation of the intersection sequence of $\alpha$ obtained by graphing the intersection sequence as a function $\sigma: \{1,2,\ldots, d-2\}^N \rightarrow \mathbb{N} $. A point $(x_i,y_i)$ of the dot graph represents an intersection of $v_{y_i}$ and $\alpha$ occurring ${x_i}$-th along $\alpha$. See Figure \ref{fig:single_dot_graph}.
\end{definition}

\begin{figure}[ht]
  \centering
  \includegraphics[width=3in]{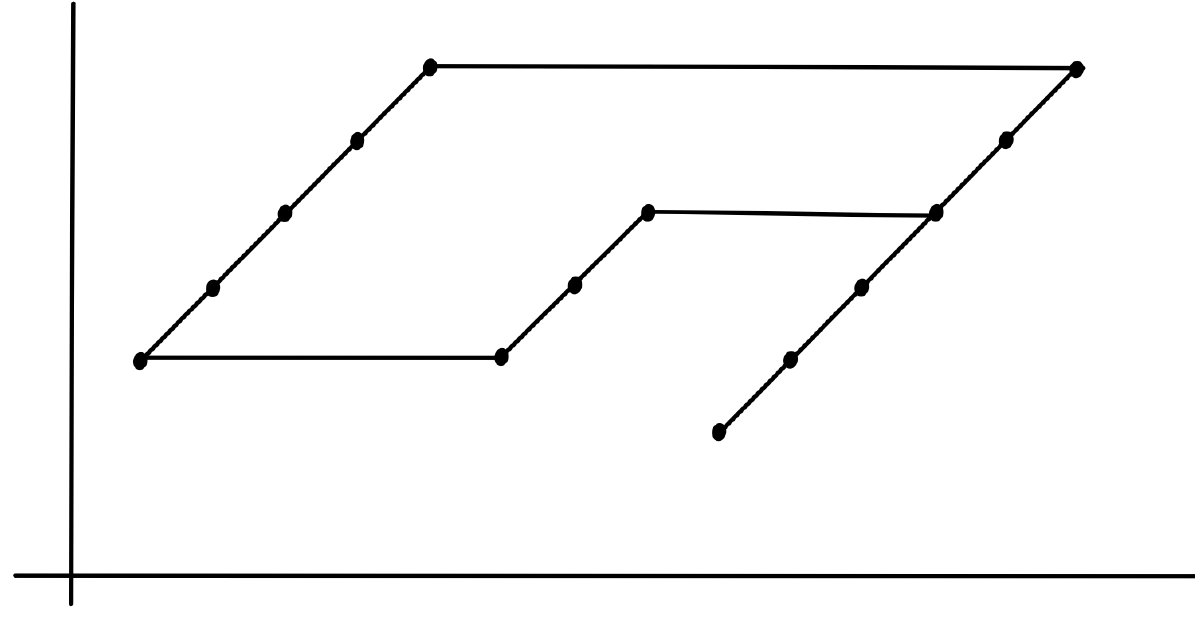}
  \caption{An example of a dot graph, with diagonal segments connected by horizontal segments to enclose a region. This dot graph represents the intersection sequence given in Figure \ref{fig:sevenfold-surgery}-left.}
  \label{fig:single_dot_graph}
\end{figure}

\begin{definition} [Empty, unpierced regions] In a dot graph, we can connect the diagonals with horizontal segments to enclose {\em regions}.  A region in a dot graph is said to be {\em pierced} if the interior of a horizontal edge of the region intersects the dot graph, and a region is said to be {\em empty} if there are no points of the dot graph in its interior. Shown in Figure \ref{fig:BMM-fig-9-regions} are what \cite{BMM} call a box, a hexagon of Type 1, and a hexagon of Type 2 (the possible hexagons without acute exterior angles.) We call these shapes regions of the {\em correct type} because, while there are other types of empty, unpierced regions, \cite{BMM} prove that only these three shapes admit surgeries on the underlying curves that result in disjoint curve pairs remaining disjoint. 
\end{definition}

\begin{figure}[ht]
  \centering
  \includegraphics[width=6in]{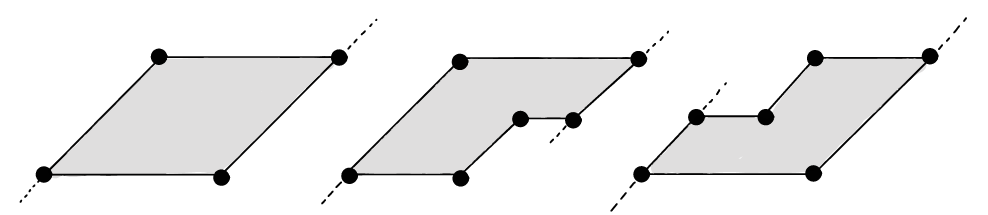}
  \caption{Examples of empty, unpierced regions in dot graphs: a box, a hexagon of Type 1, and a hexagon of Type 2. The dashed lines indicate that the dots could continue in those directions. It was shown in \cite{BMM} that only these types of regions admit adjacency-preserving surgeries on the associated curves.}
  \label{fig:BMM-fig-9-regions}
\end{figure}
The first result from \cite{BMM} that we extend in this paper explains how to use the regions of a dot graph to recognize when one can perform a set of simultaneous surgeries, as in Figure \ref{fig:sevenfold-surgery}, without causing disjoint curves to intersect. The statement in \cite{BMM} is about sequences of numbers and dot graphs representing them. We first restate that lemma in the context of the curves, and then, in the next section, extend it to our setting to include the endpoint $v$ of the path.   

\begin{lemma}[\cite{BMM} Lemma 3.4, rewritten in the language of curves] Let $\alpha$ be a reference arc for $v$, $v_1$, and $w$. Suppose that $\sigma$ is an intersection sequence for $\alpha$ that has a dot graph with an empty, unpierced box or an empty, unpierced hexagon without an acute exterior angle. Then there is a set of simultaneous surgeries one can perform on the curves crossing $\alpha$ such that the result of the surgeries is a path from $v$ to $w$ of the same length, with fewer intersections of the surgered intermediate curves with the reference arc.
  \label{lemma:bmm-reducible} 
\end{lemma}

The idea of this lemma was that, when we see empty, unpierced regions of the correct type in the dot graph, we can perform surgery on each of the curves $v_{y_i}$ that are represented by pairs of dots forming the sides of the region. It is not clear from the statement of the lemma how one chooses the type of surgery to perform, given an empty, unpierced region. That process is explained (and the freedom of choosing surgery type is revealed,) in the proof of the lemma, where \cite{BMM} defines and uses a directed graph to decide which surgeries to use (Figure \ref{fig:directed-graph-and-eg}-left.) Each vertex represents a surgery type, and the directed edge connects vertices if the surgery types, when performed on adjacent curves across the same surgery arc, result in non-intersecting curves. Starting from the bottom of the dot graph and working one's way up, one chooses any vertex of the directed graph for surgery on the bottommost pair of vertices in the dot graph, then follows any choice of the edges of the directed graph to assign the remaining surgeries to successive vertex pairs in the dot graph. Repeated applications of the lemma will reduce intersections until the path is efficient. For more details, see \cite{BMM}, pp1271-1273.  As always, an example is valuable. See Figure \ref{fig:directed-graph-and-eg}-right, which gives a graphical representation of Figure \ref{fig:sevenfold-surgery}.
\begin{figure}[ht]
  \centering
  \includegraphics[width=2.5in]{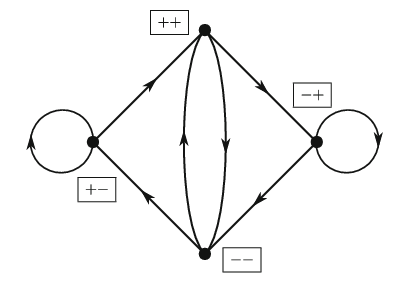}
  \hspace{.4in}
  \includegraphics[width=3in]{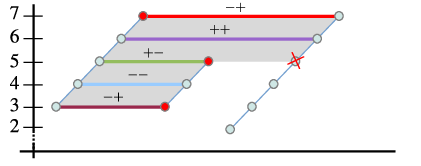}
  \caption{Left: The directed graph used to determine the surgeries in Lemma \ref{lemma:bmm-reducible}. Right: The dot graph with surgeries indicated that represent the surgeries in Figure \ref{fig:sevenfold-surgery}. Figures reproduced from \cite{BMM}. }
  \label{fig:directed-graph-and-eg}
\end{figure}

In the next section, we will extend these surgeries and redefine all of these tools to include the endpoints of a geodesic. 

 \section{Spiral surgery} \label{s:spiral surgery} 

In this section, we study a new geometric realization of the $+-$ and $-+$ surgeries, which we call {\it spiral surgery} and apply to the endpoints of a geodesic.  Recognizing patterns called \textit{spirals} in the decomposition $\Dec{v,w}$ provides a simple condition for surgery to preserve distance. 

 \subsection{A motivating example} \label{subsection:motivating eg}
 Figure \ref{fig:curlicue_example} depicts two pairs of curves, $(v, w)$ on the left and $(v', w)$ on the right. Both pairs are distance 3 in $\mathcal C(S_2)$.
 The curve pairs differ only by the green curve $v$ ``wrapping around" S and intersecting the red curve $w$ one more time than $v'$ does. 
 We have chosen an arbitrary orientation for our curves and labeled the arcs of the two curves cyclically between their intersection points, with the beginning of the labeling also chosen arbitrarily.

\begin{figure}[ht]
  \centering
\includegraphics[width=2.5in]{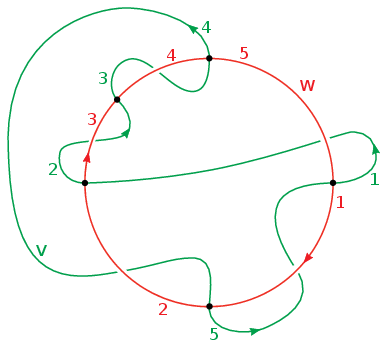}
\hspace{.5in}
\includegraphics[width=2.5in]{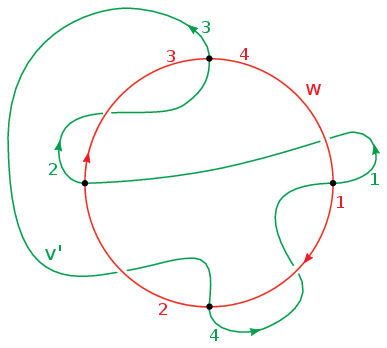}
\caption{Left: a simple curve pair $(v, w)$ with a spiral; 
  Right: the curve pair $(v', w)$ produced from a $-+$ surgery on $v$. The decomposition on the left consists of a single $12$-gon and two $4$-gons, on the right, a single $12$-gon and a single $4$-gon. In both cases, $d(v,w) = 3$. 
  }
  \label{fig:curlicue_example}
\end{figure}

In Figure \ref{fig:curlicue_example}-right, $(v', w)$ is the result of $-+$ surgery on $v$ with surgery arc $w^3$. (Note, $w^3$ is the arc of $w$ labeled $3$ in Figure \ref{fig:curlicue_example}-left.) Here we take an arc parallel to $w^3$ to be our surgery arc; it will turn out to also be a reference arc for $v$, $v_1$, $w$, where $v_1$ is any vertex adjacent to $v$ in a geodesic from $v$ to $w$.  
We cut open $v$ at the two consecutive intersection points of $v^2$  with $w$  and discard $v^2$, replacing it with a parallel copy of the surgery arc $w^3$. After relabeling so that arcs of $v$ are still consecutive, we have the curve pair on the right.  
 In this example, we remind the reader that the $-+$ surgery reduces the intersection number of a curve with its surgery arc by $1$.
Because the resulting pair still fills $S_2$ and has intersection number far below the minimum for distance $4$ on $S_2$, we know that $d(v', w)$ remains 3.  

The effect of this simple $-+$ surgery on the decomposition $\Dec{v,w}$ of the surface is the deletion of a single rectangle. The rectangle is bounded by $v^2$, $w^4$, $v^3$, and $w^3$. The surgery arc is thus parallel to one of the $w$-edges of the $4$-gon. Note that in our definition of reference arcs, we noted that for all $4$-gons and $6$-gons, all reference arcs for $v, v_1, w$ are parallel to arcs of $w$. This simplifies our considerations of surgeries on the endpoints.  

\begin{figure}[ht]
  \centering
  \includegraphics[width=6in]{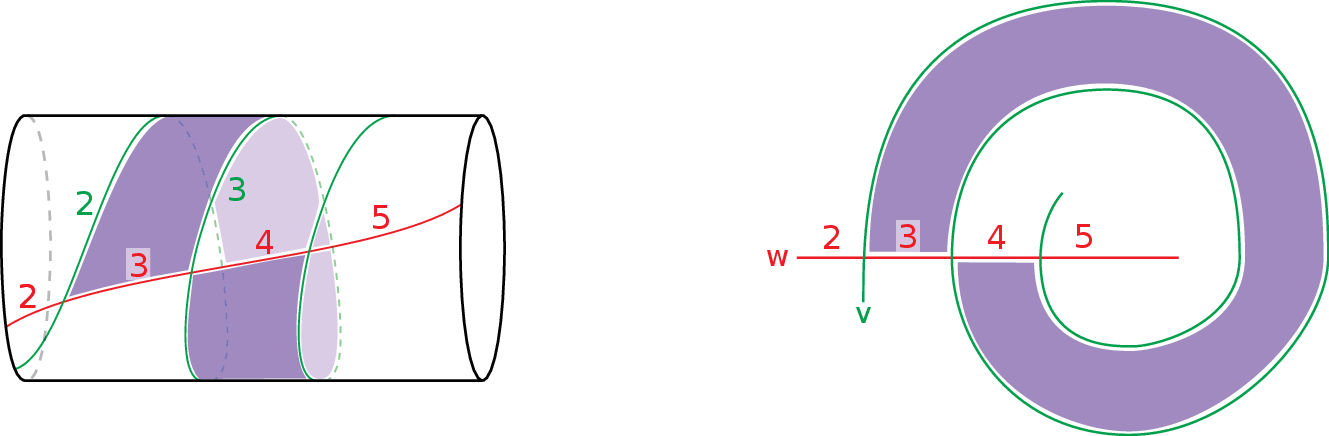}
  \caption{Recognizing a single rectangle as a spiral. Note that two corners of the rectangle are identified.}
  \label{fig:curlicue_example_with_4gon}
\end{figure}

Visualizing the arcs $v^2$, $v^3$ and $w^3$, $w^4$ on the surface, as shown in Figure \ref{fig:curlicue_example_with_4gon}-left, we see that the identified corners of the rectangle in the decomposition of $(v,w)$ means that the rectangle is wrapped around some part of the surface. When this wrapped rectangle is projected into the plane, as in Figure \ref{fig:curlicue_example_with_4gon}-left, a spiral is evident. In the next section, we define spirals to generalize and formalize this simple example of a single rectangle spiraling around the surface. 


\subsection{Spirals}\label{subsection:spirals}
{\bf Note:} In this section, we cut the closed surface open along $v$, producing a surface with two boundary components crossed by arcs of $w$. We will then be performing surgery along $v$ with surgery arcs taken to be reference arcs for $v, v_1$ and $w$. All of the arguments hold if, instead, we cut the surface open along $w$ and interchange the roles of $v$ and $w$ everywhere, and take reference arcs for $w, v_{d-1}, v$ to be the surgery arcs for $w$. In particular, our statement of Theorem  \ref{theorem:main} could be written from the perspective of surgery on $w$, studying the efficient geodesics from $w$ to $v$.  

We begin by generalizing what we observed about the rectangle in the example above.

\begin{definition}[Bands, length, width] 
  A {\bf band} of $4$-gons $\mathcal B_{v}$ in $S-v$  is a sequence of one or more $4$-gons in $\Dec{v,w}$ that are identified along $w$-edges, with its initial and final 4-gons attached along $w$-edges to $2k$-gons, $k > 2$.
  Note that the $v$ portion of $\partial \mathcal B_v$ consists of two `long' edges, each a subsequence of the cyclically ordered edges of $v$.
 The {\bf length} of a band is the number of 4-gons in $\mathcal B_v$.
  The  {\bf width} of $\mathcal B_v$ is $min\{m_v, \, i(v,w)-m_v\}$, where $m_v$ is the absolute value of the difference between the labels on the $v$-edges of any 4-gon in $\mathcal B_v$. (We define width this way to account for the cyclic ordering of the arcs of $v$.) 
  \end{definition}

The length and width of a band are independent of the choice of labeling and the choice of $4$-gon along which they are computed, because the gluing of the two boundary components of $S-v$ determines labelings up to cyclic permutation.

\begin{definition} [Spirals]
  A {\bf spiral} in $S-v$ is a band that has identified $v$-edges. The {\bf width} and {\bf length} of a spiral is the width and length of the band comprising it. See Figure \ref{fig:single_spiral_example}-left. 
  \end{definition} 
  
We note that the reference arcs across the rectangles forming a spiral are parallel to the ``spokes" formed by the  $w$-edges that cross the spiral. 
  
  \begin{definition}[Spiral interior]
  The {\bf interior} of a spiral is the set of rectangles that are bounded on all four sides by other rectangles in the spiral. The {\bf border} rectangles of a spiral are those that are not in the interior of the spiral.  
  \end{definition}
  
  In Figure \ref{fig:single_spiral_example}-right, the border rectangles are shaded yellow.   

\begin{figure}[!ht]
  \centering
  \includegraphics[width=2.5in]{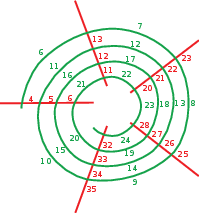}
  \hspace{1cm}
  \includegraphics[width=2.5in]{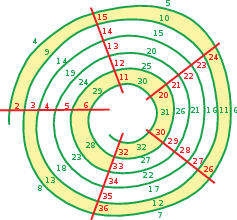}
  \caption{ Left: A spiral as a planar projection of a $4$-gon band wrapping around $S$. It has length $14$ and width $5$. Right: A spiral with the border rectangles shaded. It has length $24$ and width $5$.}
  
  \label{fig:single_spiral_example}
\end{figure}

Suppose we have a pair of vertices $v,w$ in $\mathcal C(S)$ with very large intersection number.
From Section \ref{ss:decomposition}, we know that the number of $2k$-gons, $k>2$, is finite in $\Dec{v,w}$. Thus, large intersection numbers will be manifested by a large number of $4$-gons in $\Dec{v,w}$, and as $F_4\to\infty$, we can expect to see spiraling. Conversely, surgery to reduce intersection number will be possible naturally along these bands. 

It is possible that multiple bands will form spirals together, as in Figure \ref{fig:k_spiral_example}. We call such configurations {\em $k$-spirals}, where $k$ is the number of bands traveling together in parallel. At this stage in the work, we have only defined spiral surgery for $1$-spirals, and the definitions and results do not apply to $k$-spirals, but we expect that $k$-spirals for $k> 1$ would be a more common occurrence in higher intersection number examples than the spirals of this paper. This would be a fruitful direction for further study, particularly for students. In fact, the reduction in intersection number of the Hempel example in Section 6 was obtained by repeated $+-$ and $-+$ surgeries along such $k$-spirals. 

\begin{figure}[!ht]
  \centering
  \includegraphics[width=3in]{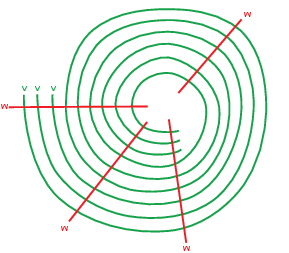}
\caption{ A $3$-spiral}
  
  \label{fig:k_spiral_example}
\end{figure}

\subsection{Spiral surgery}\label{subsection:spiral surgery}
As we studied many examples of distance 4 curve pairs on a genus 2 surface, we started to recognize a surgery that reduces $i(v,w)$: surgery on arcs of the curve in the interior of a spiral. Such surgery removes an arc of $v$ from the interior of the spiral while leaving the rest of the curve (and thus the rest of the decomposition,) fixed. We found this surgery appealing because the effects of this surgery are local, contained entirely within the spiral, so we could analyze the new curve more easily and hope to see the distance preserved. 

\begin{definition}[Spiral Surgery] 
  Let $\mathcal{B}$ be a spiral. {\em Spiral surgery} is a $(+-)$- or $(-+)$-surgery on $v$ with surgery arc parallel to a $w$-edge on an interior rectangle of the spiral.  \end{definition}
  \begin{figure}[!ht]
  \centering
  \includegraphics[width=2.5in]{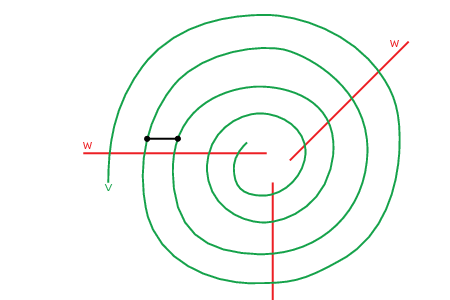}
  \hspace{1cm}
  \includegraphics[width=2.5in]{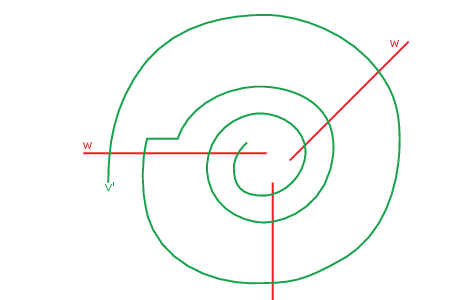}
  \caption{An example of spiral surgery. Left: a spiral in a curve pair $(v,w)$, in green and red respectively, and the surgery arc in black. Right: The resulting curve pair $(v', w)$ after  spiral surgery. The decomposition has three fewer $4$-gons. Outside of this figure, the curves, and thus the decomposition, are unchanged, so $i(v',w) = i(v,w) - 3$. }
   \label{fig:spiral surgery:before and after}
 \end{figure}
  An immediate question is why we have restricted our attention to $-+$ or $+-$ surgeries. In [BMM], there were 4 types of surgery performed on intermediate curves in a path: $++$, $- -$, $+-$ and $-+$. All of them were required in order to make a path efficient. 
In our work, we discovered that when we perform surgery on $v$ across rectangles in $\Dec{v,w}$, only $+-$ and $-+$ surgery preserve the distance between $v$ and $w$:  
\begin{proposition}\label{proposition:surgery restrictions} Let $v,w \in \mathcal C(S)$ be such that $d(v,w) \geq 3$.  Let $\gamma$ be a reference arc contained in a rectangle in $\Dec{v,w}$. Then type $++$ or $--$ surgery on either $v$ or $w$ relative to the surgery arc $\gamma$ will not preserve $\Dec{v,w}$. 
\end{proposition}

\begin{proof}
 Consulting \Cref{fig:surgery}, we see that, if we are performing $++$ or $--$ surgery across a rectangle, the two strands of the to-be-surgered curve $v$ in the leftmost sketch of \Cref{fig:surgery} must have opposite orientations and the result of surgery will be a choice of one of two simple closed curves.  
Without loss of generality, suppose that we performed a surgery of type $++$ across $\gamma$ in a rectangle $r$ that is part of a band of $\geq 1$ rectangles joining $w$-edges $6$-gons $P, Q \in \Dec{v,w}$, where it is possible that $P = Q$.
\Cref{fig:++ or --}-top depicts the setup.

 \begin{figure}[!ht]
   \centering
  \includegraphics[width=.25\textwidth]{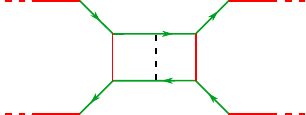}

  \vspace{.2in}

  \includegraphics[width=.25\textwidth]{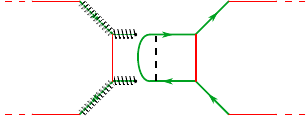}

  \vspace{.2in}

  \includegraphics[width=.1\textwidth]{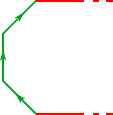}
  \caption{++ or $--$ surgery across a rectangle will not preserve $\Dec{v,w}$.}
   \label{fig:++ or --}
 \end{figure}
 
After the surgery, $v$ will be changed to two distinct closed curves, $v'$ and $v^{\prime\prime}$ and (without loss of generality) we choose one of them, say $v'$, so our new pair is $(v',w)$. We will show that the pair $(v',w)$ no longer fill a surface of genus $S$. \Cref{fig:++ or --}-middle shows $v'$ and focuses on the polygon $P$. 
The surgery creates a bigon in $(v',w)$ and, after pushing $v'$ across the bigon,  there will be a new bigon if there is another rectangle in the band, and the argument can be repeated. We continue to push $v'$ across the rectangles to eliminate bigons. 
The final bigon will be formed by $v'$ and a $w$ edge of $P$.
After removal of this final bigon the polygon $P$ will have changed to a polygon $P'$ with 2 fewer edges, as in \Cref{fig:++ or --}-bottom.  Thus $F_4$ has increased by 1 and $F_6$ has decreased by 1, while no other polygons have been created.
That means that $(v',w)$ no longer fill a surface with the given genus, by the equations in Section \ref{ss:decomposition}, so our surgery of type $++$ is not possible. The argument for $n$-gons greater than 6 and for $--$ surgery is identical. 
\end{proof}

See Figure \ref{fig:spiral surgery:before and after} for an example of a spiral surgery, where we have marked a reference arc at $9$ o'clock for our surgery arc. 
Here the surgery is $-+$ surgery. One piece of $v$ in the spiral that begins and ends on the endpoints of the surgery arc is removed, then replaced with a parallel copy of the surgery arc, leaving $v$ otherwise unchanged. Generally, the sign of a spiral surgery ($+-$ or $-+$,) is determined only by the direction the spiral winds on the surface, independent of orientation. 

Spiral surgery removes several $4$-gons from the spiral; the number of $4$-gons removed is equal to the width of the spiral. The rest of the decomposition is unchanged. Thus, spiral surgery reduces $i(v,w)$ by the width of the spiral. 

Looking back at Figure \ref{fig:directed-graph-and-eg}-left, We have learned from Proposition \ref{proposition:surgery restrictions} that, in spiral surgery, we must begin at the bottom of any extended dot graph with surgery of the form $-+$ or $+-$ and go from there. However, one cannot simply perform spiral surgery whenever one spots a spiral.

\begin{example} Let's examine an example of a curve pair $(v,w)$ which has a spiral but which does not admit a spiral surgery. In Figure \ref{fig:spiral-surgery-non-possible}-left, we see portions of a curve pair $v$ and $w$ of distance at least $5$, with intermediate curves of an efficient geodesic drawn in the spiral. Here we have some of the intermediate curves of the geodesic ``crossing over" the spiral rather than traveling in parallel along $v = v_0$ (labeling curves $v_k$ by their index $k$.) In Figure \ref{fig:spiral-surgery-non-possible}-right, we have the concatenated dot graphs of two adjacent intersection sequences on an interior loop of the spiral, using the consecutive edges of $w$ forming a spoke of the spiral as our reference arcs. These dot graphs appear along the upper right, lower right, and bottom arcs of $w$ across the spiral.   Examining the regions drawn, we recognize that no surgery is possible that will preserve the adjacency of vertices in the curve graph. This is because for a box region, the upper left vertex is required to be a terminal vertex in the ascending edge of the dot graph, which is not the case for any of these regions.

\begin{figure}[!ht]
  \centering
  \includegraphics[width=3in]{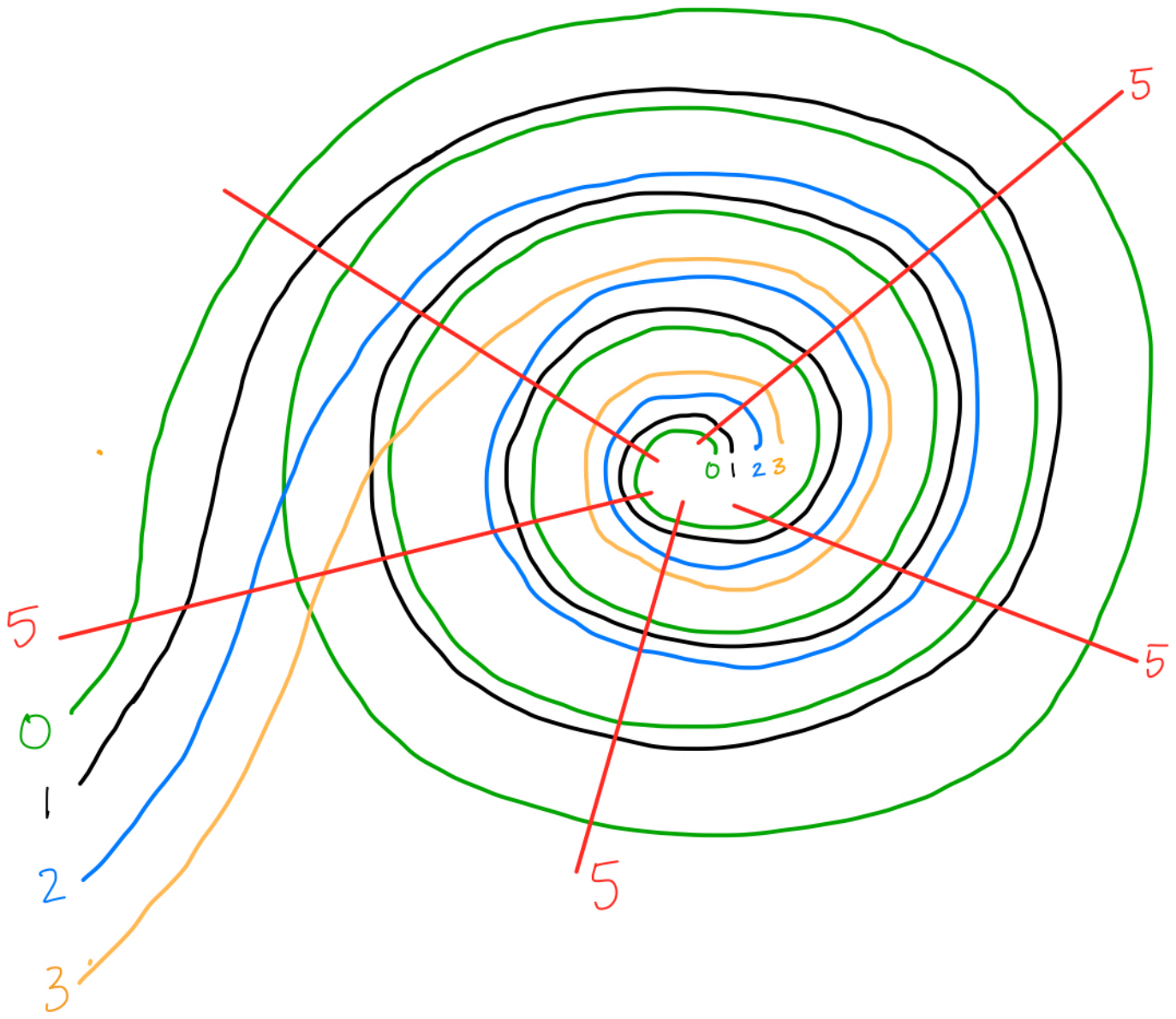}
  \hspace{1cm}
  \includegraphics[width=2.5in]{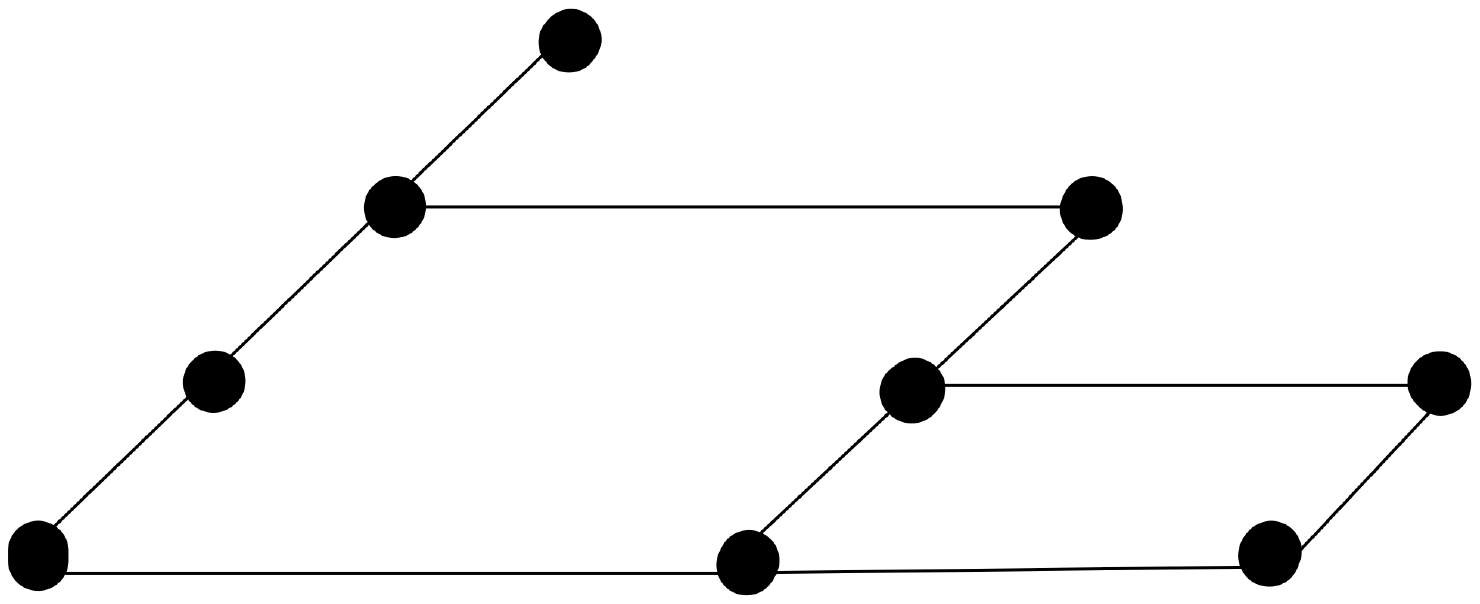}

  \caption{Example \ref{eg:spiral-surgery-non-possible}, a spiral in which spiral surgery is not possible. }
   \label{fig:spiral-surgery-non-possible}
 \end{figure}
 
 \label{eg:spiral-surgery-non-possible}
\end{example} 

This example reminds us that there is critical information in the dot graph, and suggests that in the case of surgery on $v$, we need to consider dot graphs that are concatenated to get the information we need. We explore this idea further in the next section. 
 
\section{When can we perform surgery on the endpoints of a geodesic? \label{s:reduce-i}}

In this section, we extend the concepts of intersection sequences, dot graphs, and regions of the correct type to incorporate the endpoints of a path. We define these tools in the most generality possible, but our application in this paper will be only to the situation of surgery across spirals. Our immediate goal is to find a way to include the endpoints of the reference arcs (that is, where $v$ crosses the reference arc,) in a dot graph. 

We begin by defining our construction for a single fixed efficient geodesic $ v=v_0, v_1,v_2, \ldots, v_d = w$ in $\mathcal C(S)$ (though the construction works generally for efficient paths.)  Then we will apply it to the finite set of all efficient geodesics from $v$ to $w$, $\{\mathcal G_1, \mathcal G_2, \ldots, \mathcal G_N\}$. We note that the decomposition $\Dec{v,w}$ and the set of reference arcs for $v,v_1,w$ will be the same for each geodesic $\mathcal G_j$. We will denote the reference arcs with superscripts, by $\gamma^i, \gamma^j, \gamma^k$, etc. 

\subsection{Extending dot graphs to include the endpoints of a geodesic \label{ss:extended-dot-graphs}}

\begin{definition}[Consecutive reference arcs] We say that two reference arcs (relative to $v$) are {\bf consecutive} if one endpoint of each arc lies on the same edge of $v$ while the reference arcs themselves lie on opposite sides of the edge. See Figure \ref{fig:consec-ref-arcs-defn} below. 
\end{definition}
\begin{figure}[ht]
  \centering
  \includegraphics[height=1.8in]{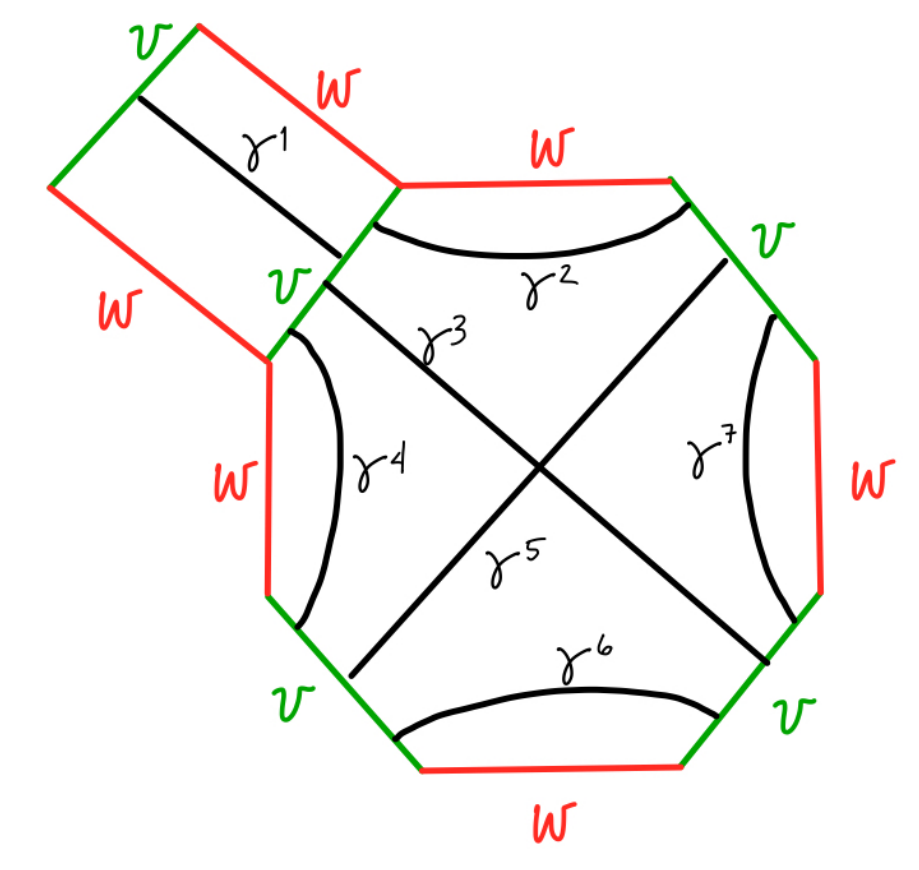}
    \hspace{2cm}
  \includegraphics[height=1.8in]{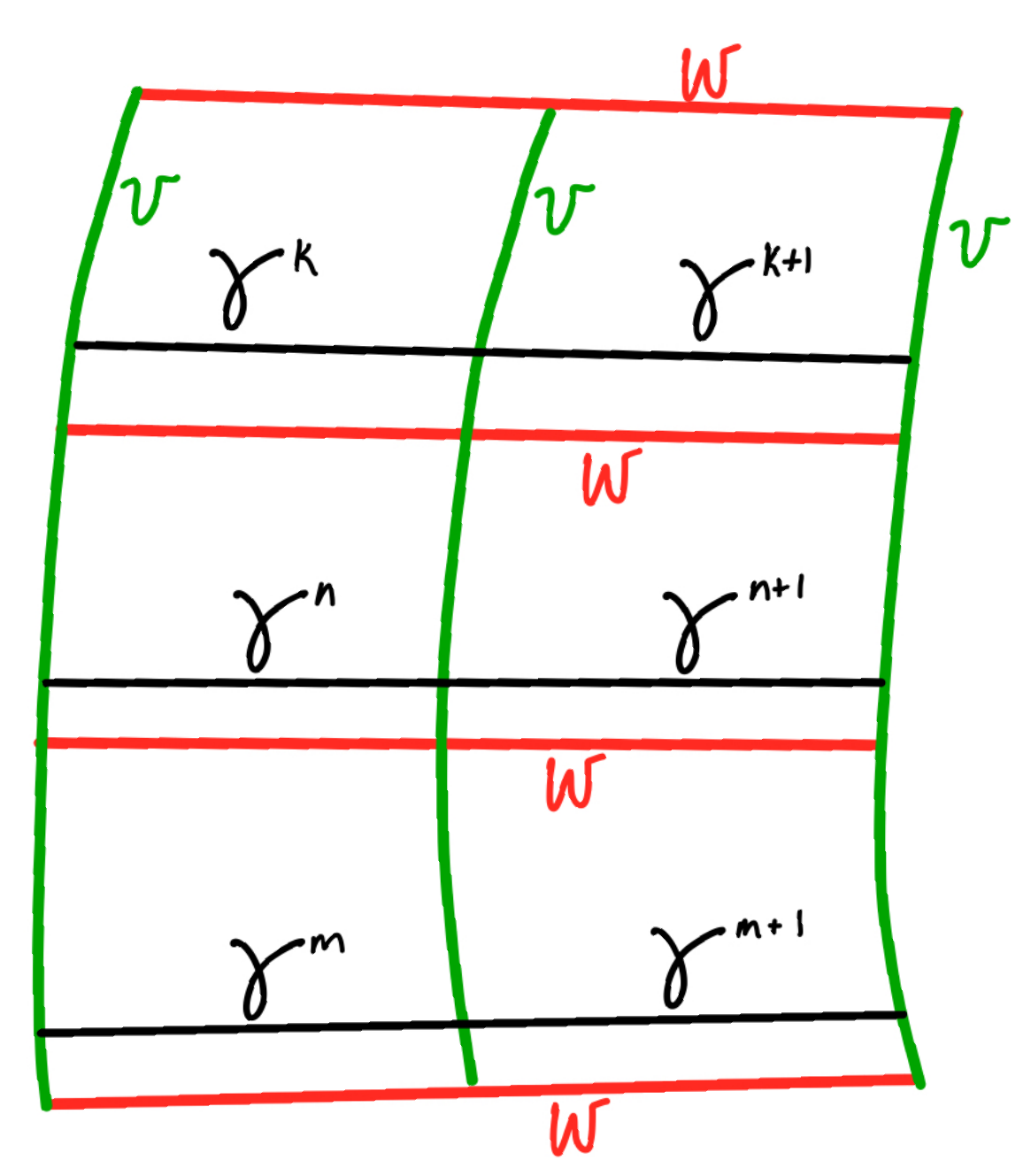}
    \caption{Left: The reference arcs $\gamma^1$ and $\gamma^2$ are consecutive, as are the pairs $\gamma^1, \gamma^3$ and $\gamma^1, \gamma^4$.  Right: consecutive reference arcs across $4$-gons that share edges as in a spiral.}
  \label{fig:consec-ref-arcs-defn}
\end{figure}

\begin{definition}[Extended intersection sequence] Fix a pair of consecutive reference arcs $\gamma^k, \gamma^m$ and let $sigma_k, \sigma_m$ be their intersections sequences with respect to some length $d$ path from $v=v_0$ to $w=v_d$.
 The {\em extended intersection sequence} $\sigma_{k,m}$  is the sequence of non-negative integers marking the order of intersections of $v_1, \ldots,v_{d-2}$ with $\gamma^k \cup \gamma^{m}$ with respect to some orientation. It is obtained by first adding a $0$ to the beginning of the intersection sequence $\sigma_k$ and the end of the intersection sequence $\sigma_{m}$, then joining these two sequences with a $0$ at the end of $\sigma_k$ and the beginning of $\sigma_{m}$. 
\end{definition}
 The appearances of $0$ in the new sequence correspond to the three intersections of $v$ with $\gamma^k \cup \gamma^{m}$. We note that in the initial definition of reference arcs, an arbitrary orientation is taken. When we concatenate the intersection sequences of two adjacent reference arcs, then, it may be the case that the chosen orientations are not the same, but we want to travel in the same direction along a concatenated pair. Thus, the extended intersection sequence could consist of any of the following (where $\overline{\sigma}$ means the reverse of sequence $\sigma$): 
 \begin{align}
\sigma_{k,m} & = (0,\sigma_k,0,\sigma_{m},0) \notag \\
         & = (0,\overline{\sigma_k},0,\sigma_{m},0) \notag \\
         & = (0,\sigma_k,0,\overline{\sigma_{m}},0) \notag \\
         & = (0,\overline{\sigma_k},0,\overline{\sigma_{m}},0)
\end{align}
We note that in the case of consecutive reference arcs across a spiral, we can take the reference arcs to follow the orientation of $w$ given by its labeling, so we are always in the first or last case.

Since the permutations to get a sequence into sawtooth form require only that the entries have indices differing by more than $1$, we can extend the process of putting a sequence into sawtooth form from sequences of natural numbers to sequences of non-negative integers. Once we have constructed an extended intersection sequence of consecutive arcs $\gamma^k$, $\gamma^{m}$, then, we put the extended sequence into sawtooth form by means of transposing curves that are not adjacent in the path, without changing the adjacency of any curves in the curve graph.  
We now can include the curve $v = v_0$ in the dot graph:  
 
 \begin{definition}[Extended dot graphs] Starting with the sawtooth-form extended intersection sequence $\sigma_{k,m}$,  we construct a new dot graph over the consecutive reference arcs $\gamma^k \cup \gamma^{m}$. We call the dot graph of the sawtooth-form extended intersection sequence the {\bf extended dot graph} of $\sigma_{k,m}$.
  
\end{definition}
We note that we don't simply concatenate the dot graphs of $\gamma^k$ and $\gamma^{m}$. Even if the orientations coincide, we first put the extended intersection sequence into sawtooth form, and then we create the extended dot graph. 

\begin{example} \Cref{fig:extended_dot_graph} shows an example; suppose the individual dot graphs for the sequences for a pair of consecutive reference arcs with different orientations are $\sigma_k = (1,2,3,2,3,4) $, $\sigma_{m} = (2,3,4,5,1,2,3,4)$, so the extended intersection sequence is $(0,1,2,3,2,3,4, 0, 2,3,4,5,1,2,3,4,0)$. After transforming into sawtooth form, the intersection sequence is $$\sigma_{k, m}=(0,1,2,3,2,3,4,2,3,4,5,0,1,2,3,4,0)$$ and we can construct the extended dot graph.  \end{example}

\begin{figure}[ht]
  \centering
  \includegraphics[height=1in]{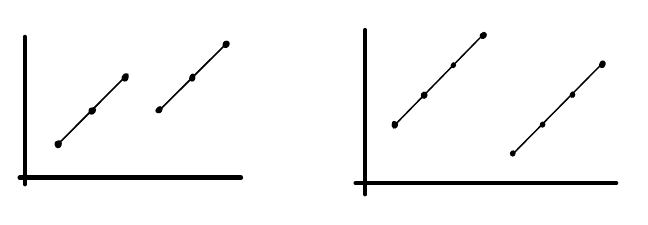}
  \hspace{.5in}
  \includegraphics[height=1in]{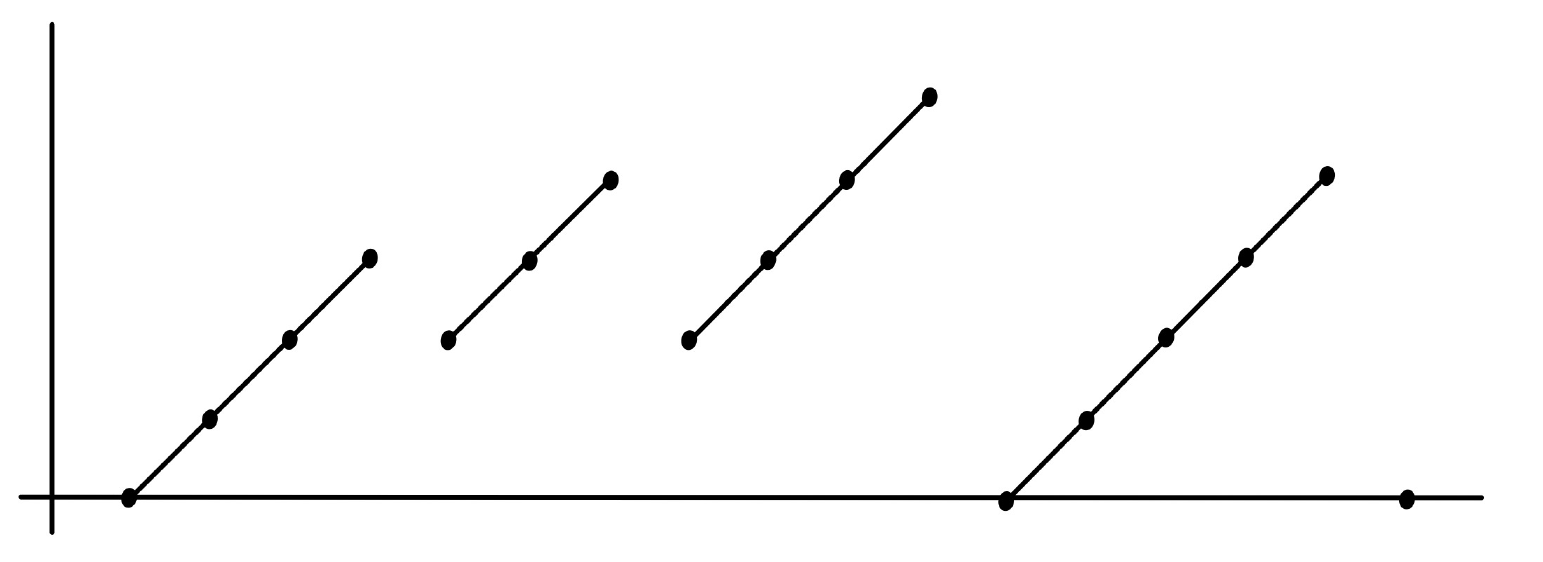}
  \caption{As described in Example 4.4: Left: Dot graphs of $\sigma_k$, $\sigma_{m}$. Right: Extended dot graph of $\sigma_{k,m}$ }
  \label{fig:extended_dot_graph}
\end{figure}

We connect vertices with slope-$1$ and slope-$0$ edges in the extended dot graph to enclose regions of the plane just as we did with dot graphs.
We define the empty and unpierced regions the same way in extended dot graphs as we did in the original dot graphs. We note that we recognize the pair $0,0$ in the extended dot graph of an extended intersection sequence as a trivial ``region", where we have permuted an entire sequence $\sigma_k$ or $\sigma_{m}$ past $0$.   

New empty, unpierced regions may appear in the extended dot graph, while there may be no empty, unpierced regions in the dot graphs lying above the individual $\gamma^i$ for any $i$. An example of this is seen in Figure \ref{fig:spiral-surgery-non-possible} in the previous section. In fact, we will assume that when we start, we have already performed all possible surgeries across any empty, unpierced regions of the correct type in any of the individual dot graphs above the $\gamma^k$.

Suppose a new empty, unpierced region of the correct type appears in the extended dot graph. If the lower corners of the regions are at height $0$, then the simultaneous surgeries (as described in \cite{BMM} and pictured in Figure \ref{fig:sevenfold-surgery})  to remove the region would involve the curve $v = v_0$. 
In this case, the next lemma extends Lemma \ref{lemma:bmm-reducible} so that we can perform surgery on $v$ and simultaneously on the intermediate curves appearing in the extended dot graph. We state it for paths, rather than for geodesics. 

\begin{lemma}[Reducing the intersection number of endpoints of an efficient path] \label{lemma:reducible}  For an efficient path $\mathcal{P} = (v, v_1, \cdots, v_{d-1}, w)$ from $v$ to $w$, suppose that there is a pair of consecutive reference arcs $\gamma^k$, $\gamma^{m}$ such that the extended dot graph of the extended intersection sequence $\sigma_{k,m}$ has an empty, unpierced box or hexagon of the correct type. 
 Let $v'_i$ be the result of surgery on $v_i$ along a surgery arc parallel to $\gamma^k$, $i=0,1,2, \ldots, d-1$ (if $v_i \cap \gamma^k = \emptyset$, then $v'_i = v_i$).
 Then surgery transforms the path $\mathcal{P}$ into a new path $\mathcal{P'}$ from $v'$ to $w$,  $\mathcal{P'} = (v', v_1', \cdots, v'_{d-1}, w)$, with $i(v',w) \leq  i(v,w)$ and the length of $\mathcal{P}$ the same as the length of $\mathcal{P'}$.

  \end{lemma}
  
\begin{proof} 
 Suppose we have an empty, unpierced region of one of the correct type. We assume that the bottom edge of the region is at height 0, since otherwise we are in the setting of Lemma \ref{lemma:bmm-reducible}.  We choose the surgery type appropriate to $v_0=v$, based on the configuration of $v$ relative to the surgery/reference arc $\gamma^{k,m}$. We then simultaneously perform surgery on each of the $v_{j_i}$, where $j_i \in \sigma_{k,m}$ appears as a pair of points in the boundary of the region, as determined by the directed graph in \Cref{fig:directed-graph-and-eg} . 
  
 All of the surgeries preserve or reduce intersection number; by construction, no new intersections are created. By following the directed graph \Cref{fig:directed-graph-and-eg} in choosing successive surgeries for the curves, we preserve disjointness of all of the surgered curves, including $v$ with $v_1$. 
Therefore $i(v',w) \leq i(v,w)$.
 Moreover, surgery preserves essentialness.
 We conclude that the resulting set of curves $v'_1, v'_2, \ldots, v'_{d-1}$ represent a path $\mathcal{P'}$ from $v'$ to $w$ in the curve graph of length $d$.
 \end{proof}

In Lemma \ref{lemma:reducible}, we didn't claim that if $\mathcal{P}$ were a geodesic between $v$ and $w$ then $\mathcal{P'}$ would be a geodesic from $v'$ to $w$. 
Here is why:  Lemma \ref{lemma:reducible} only explains what would happen in {\em one} efficient geodesic. 
That is, the extended dot graph only tells us where we might do surgery on $v$ and simultaneous surgeries on the $v_k$ that result in $i(v',w) < i(v,w)$ and a length $d$ path from $v'$ to $w$. 
However, changing the vertices $v,w$ to a new pair $v',w$ affects {\em every} efficient geodesic with endpoints $v, w$. For some other efficient geodesic $\mathcal G_r$, the extended dot graph above the same pair of consecutive reference arcs might have no empty, unpierced regions of the correct type. Performing surgery on $v$ along those reference arcs could produce a curve $v'$ that is a smaller distance from $w$. Therefore, showing that a surgery on $v$ preserves $d(v,w)$ requires us to show that (1) the surgery preserves the length of {\em every} efficient geodesic from $v$ to $w$ as a path from $v'$ to $w$ and (2) doesn't create any new, shorter paths between $v'$ and $w$ in $\mathcal C(S)$. 

We defined the extended dot graph for any efficient path in $\mathcal C(S)$, but to consider the extended dot graph for {\em every} efficient geodesic, we will need to use the finiteness of the number of efficient geodesics that is given by Theorem \ref{thm:bmm-main}.

\begin{definition}[Stacked extended dot graph] Since the same curves $v,w$ are the endpoints of each geodesic $\mathcal G_i$, with the same set of reference arcs $\gamma^k$, we can consider the extended dot graph of consecutive reference arcs $\gamma^k, \gamma^{m}$ for each of the $N$ efficient geodesics $\{\mathcal G_1, \mathcal G_2, \ldots, \mathcal G_N\}$ from $v$ to $w$. The {\bf stacked extended dot graph} of $\gamma^k, \gamma^{m}$ is the set of all $N$ extended dot graphs over $\gamma^k, \gamma^{m}$. We will denoted the set by \D$_{k,m}(v,w)$. 
\end{definition}

\begin{definition}[Regions in stacked extended dot graphs] Let \D$_{k,m}(v,w)$ be a stacked extended dot graph of $\gamma^k$ and $\gamma^{m}$.
 If there is an empty, unpierced region of the correct type (box or hexagon with no acute exterior angles,) in the extended dot graph of $\sigma_{k,m}$ for each $\mathcal G_i \in \{\mathcal G_1,\mathcal G_2, \ldots, \mathcal G_N\}$, then we say that there is an {\bf empty, unpierced region of the correct type} in \D$_{k,m}(v,w)$.

\end{definition}
 
 \begin{figure}[ht]
  \centering
  \includegraphics[width=2.5in]{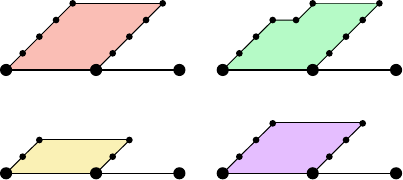}
  \hspace{.5in}
  \includegraphics[width=2.5in]{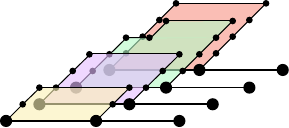}
  \caption{Left: extended dot graphs of $\sigma_{k,m}$ for four different paths from $v$ to $w$. Shaded in for each is an empty, unpierced region. Right:  \D$_{k,m}(v,w)$, the stacked extended dot graph for the four paths.  }
  \label{fig:stacked_extended_dot_graph}

\end{figure}

In \Cref{fig:stacked_extended_dot_graph}, we have sketched a simplified example of  \D$_{k,m}(v,w)$ for four different paths. The intuition here is that if one could ``see" through a hole all the way through a stacked extended dot graph of some $\sigma_{k,m}$, then one could perform a spiral surgery that preserves each of the efficient geodesics as paths. Indeed, with the guidance of the stacked extended dot graph \D$_{k,m}(v,w)$, we can perform simultaneous surgeries on all of the efficient geodesics between $v$ and $w$ as indicated by the presence of an empty, unpierced region of the correct type above the surgery arc $\gamma^k$. Iterating Lemma \ref{lemma:reducible}, we see that the result will be a set of efficient paths from $v'$ to $w$ of the same length as the distance from $v$ to $w$. This discussion yields the following: 

\begin{proposition}
Suppose that $v$ and $w$ are filling curves of distance $d$ on a surface $S$, with $N$ efficient geodesics from $v$ to $w$, and suppose that \D$_{k,m}(v,w)$ has an empty, unpierced region of the correct type above the reference arc $\gamma^k$. Then the set of $N$ efficient geodesics is sent to a set of efficient paths connecting $v'$ to $w$ by means of simultaneous surgery, where $v'$ is the result of performing surgery on $v$. \qed
\end{proposition}

\subsection{Applying stacked, extended dot graphs to spiral surgery \label{ss:main-theorem}}

In this section, we apply the general machinery of extended dot graphs to spirals. The key observation is that, in the interior of a spiral, adjacent $v=v_0$ vertices that are adjacent in the dot graph are actually endpoints of a single loop of the spiral. This constrains the extended dot graphs that can appear over reference arcs of a spiral. 
\begin{figure}[ht]
  \centering
  \includegraphics[height=1in]{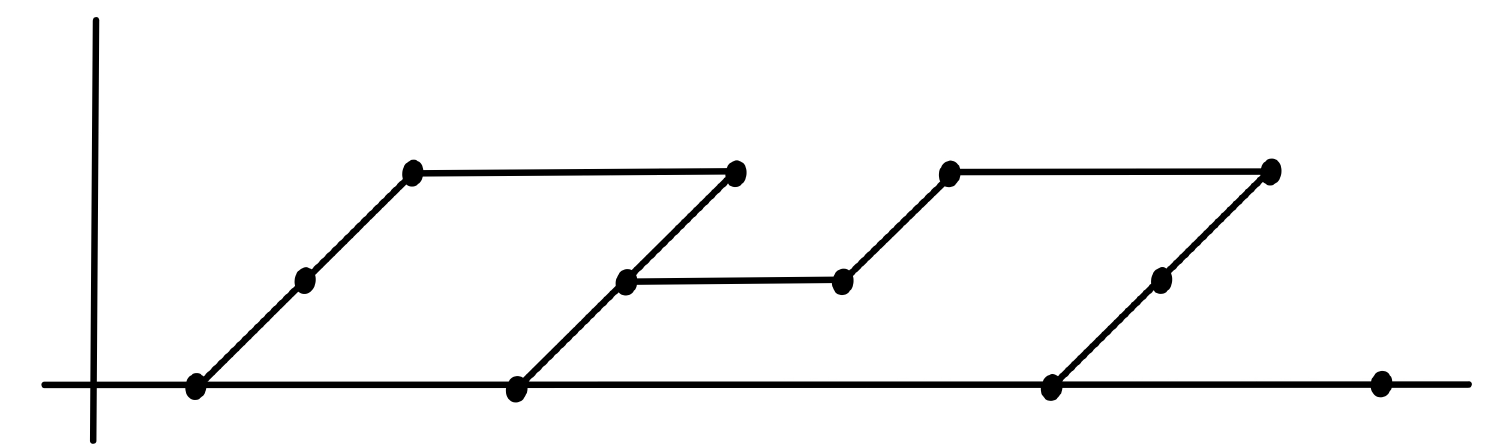}
 
  \caption{Adjacent extended dot graphs that could not appear in a spiral.}
  \label{fig:box-hex-dot-graph}
\end{figure}
For example, consider the adjacent extended dot graphs in the Figure \ref{fig:box-hex-dot-graph}. Both of these extended dot graphs are empty, unpierced, and (unlike those in Example  \ref{eg:spiral-surgery-non-possible},) of the correct type. However, if they appeared on the interior of a spiral, the configuration of curves $v=v_0, v_1, v_2$ that they represent could not be disjoint. For the intersection sequence represented by Figure \ref{fig:box-hex-dot-graph} is given by $$0,1,2,0,1,2,1,2,0,1,2, 0,$$ and the interior piece of the spiral this would describe is shown in Figure \ref{fig:impossible-spiral}.

\begin{figure}[ht]
  \centering
    \includegraphics[height=2in]{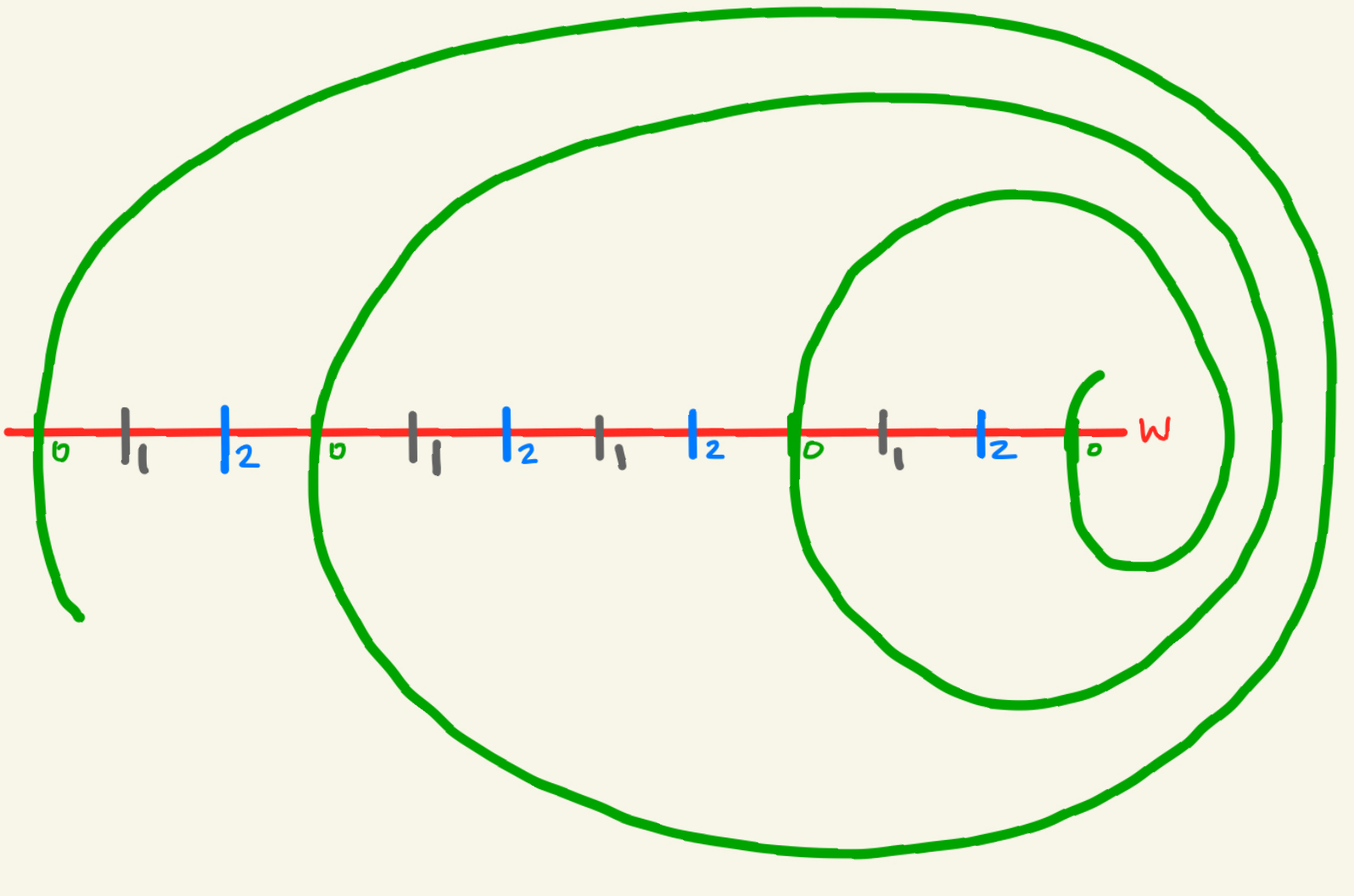}
  \caption{A spiral that has the extended dot graphs given by   Figure \ref{fig:box-hex-dot-graph}. The arc of $v$ in the spiral is drawn in its entirety, while only the intersections of $v_1$ and $v_2$ with $w$ are indicated. One sees immediately that there is no configuration of $v,$, $v_1$< and $v_2$ that has these intersections with $w$ that will keep curves $v$ and $v_1$ disjoint, and $v_1$ and $v_2$ disjoint.}
  \label{fig:impossible-spiral}
\end{figure}

This example, as well as Example \ref{eg:spiral-surgery-non-possible}, leads us to the following lemma: 

\begin{lemma} Let $v$ and $w$ be a filling pair on $S$ such that the curves form a spiral on $S$. If the interior of the spiral admits a spiral surgery, that is, its extended dot graphs have empty, unpierced regions of the correct type, then the spiral has identical extended dot graphs across its interior. 
\label{lemma:spiral-dot-graph}
\end{lemma}

\begin{proof}

The proof follows case by case, taking the possible pairings of the three types of regions that on their own might admit surgeries: box + hexagon of Type I, box + hexagon of Type II, and hexagon of Type I + hexagon of Type II, and the same pairings in opposite order. By examining the spiral with the configuration of curves represented by the extended intersection sequence, as in Figure \ref{fig:impossible-spiral}, one immediately sees that the only realizable curve configurations  occur when the dot graphs on the interior are identical.  
\end{proof}

This paper's main result is that, under the conditions outlined in this section, spiral surgery preserves distance while reducing intersection number. 

\begin{theorem}[Spiral surgery preserves efficiency and distance]
\label{theorem:main}
 Let $v,w$ be vertices in $\mathcal C(S)$, $S$ closed and genus $g$, with distance $d=d(v,w) \geq 3$. 
Suppose that $\Dec{v,w}$ has a spiral of width $m$, where $m < i(v,w) -  i_{min}(d, g, \mathcal{F})$, and suppose that there is an empty, unpierced region of the correct type in \D$_{k,k+1}(v,w)$ over the interior of the spiral. Let $v'$ be the result of the indicated spiral surgery on $v$.
 
Let $\{\mathcal G_1, \mathcal G_2, \ldots, \mathcal G_n\}$ be the set of all efficient geodesics connecting $v$ and $w$ in $\mathcal C(S)$ and let $\{\mathcal G'_1, \mathcal G'_2, \ldots, \mathcal G'_n\}$ be the set of paths resulting from the simultaneous surgeries on each $\mathcal G_j$ as indicated by each extended dot graph. Then $i(v',w) < i(v,w)$ and each $\mathcal G'_j$ is an efficient path of length $d(v,w)$ from $v'$ to $w$. Furthermore, $d(v', w) = d(v,w).$
\end{theorem}

\begin{proof}
{\bf Spiral surgery preserves efficiency:}
  Each $\mathcal G'_j$ is a path of length $d$, due to repeated applications of Lemma \ref{lemma:reducible}. 

  We recall that a path is efficient if the oriented sub-geodesic $v_k,\dots,v_d=w$ is initially efficient for each $0\leq k\leq d-3$ and the oriented path $v_d, v_{d-1}, v_{d-2}, v_{d-3}$ is also initially efficient.

We assume that we have performed all possible surgeries to remove any empty, unpierced regions of the correct type in the stacked extended dot graphs.

 We begin by checking that the first part of the efficient definition holds true for an arbitrary choice of $\mathcal G'_j = v', v'_1, \ldots, v'_{d-2}, v'_{d-1}, w$, the image of some efficient geodesic $\mathcal G_j = v, v_1, \ldots, v_{d-2}, v_{d-1}, w$ under spiral surgery.
  Fix $k$.
We know that the original subgeodesic $v_k, v_{k+1}, \ldots, v_{d-1}, w$ was initially efficient, so that $|v_{k+1} \cap \alpha | \leq  d-k-1$, for any reference arc $\alpha$ relative to $v_k, v_{k+1}, w$. Let $|v_{k+1} \cap \alpha | = m \leq  d-k-1$.

 If $v'_k = v_k$ and $v'_{k+1} = v_{k+1}$, then there is nothing to show, so we assume that $v'_k \neq v_k$ and $v'_{k+1} \neq v_{k+1}$. This means that $v_k$ and $v_{k+1}$ were surgered as a result of the spiral surgery. Since the location of the surgery on $v$ was contained entirely within the spiral, that means that $v_k$ and $v_{k+1}$ passed through the interior of the spiral formed by $v$ and $w$. 
By Lemma \ref{lemma:spiral-dot-graph}, we know that if $v_k$ appears in the interior of a spiral where we are doing spiral surgery, then it will form a spiral with $w$ there as well. We then can apply Proposition \ref{proposition:surgery restrictions} to the subgeodesic $v_k, v_{k+1}, \ldots, v_{d-1}, w$ to recognize that the only possible surgery on $v_k$ and $v_{k+1}$ would have been $+-$ or $-+$ surgery, matching the surgery done on $v$. 
(This observation marks an important simplifying aspect of spiral surgery: a complicated set of surgeries like that seen in Figure \ref{fig:sevenfold-surgery} will not appear in a spiral surgery. In fact, Proposition \ref{proposition:surgery restrictions} and Lemma \ref{lemma:spiral-dot-graph} tell us that the only part of the graph we would use from Figure \ref{fig:directed-graph-and-eg}-left to determine which surgeries are possible would be the loop connecting vertex $+-$ to itself and the loop connecting vertex $- +$ to itself.)

Within the interior of the spiral, spiral surgery reduces pairs of intersections of $v_{k+1}$ with some reference arc $\alpha$ for $v_k, v_{k+1}, w$ to single intersections of $v'_{k+1}$ with some reference arc $\alpha'$ for $v'_k, v'_{k+1}, w$. For the spiral surgery to have been possible on $v_{k+1}$, then, there were an even number of intersections of $v_{k+1}$ with $\alpha$ on the interior of the spiral, with the surgery then resulting in half the number of intersections. The new reference arc $\alpha'$ is either the same as some reference arc $\alpha$ for $v_k, v_{k+1}, w$ or it is a new one. If it is new, because we are working on the interior of a spiral we know that it consists of a concatenated pair of reference arcs across the spiral given by $v_k$ and $w$. Thus, since the subgeodesic from $v_k$ to $w$ was initially efficient, we can bound $|v'_{k+1} \cap \alpha' |$  by twice the bound on $|v_{k+1} \cap \alpha |$, but we also know that spiral surgery reduced these intersections by half along each of the two reference arcs. We conclude that $|v'_{k+1} \cap \alpha' | \leq  2m/2 = m \leq d-k-1$. 

Finally, we consider whether the oriented path $w= v_d, v'_{d-1}, v'_{d-2}, v'_{d-3}$ that results from the surgeries is also initially efficient.
Since the surgery arcs were all parallel to arcs of $w$ (since in a spiral, we only have $4$-gons,) and since $w$ and $v_{d-1}$ are disjoint, $v'_{d-1} = v_{d-1}$. 
If $v'_{d-3} \neq v_{d-3}$, and then we analyze the intersections of $v'_{d-1}$ with some reference arc $\alpha'$ for $w, v_{d-1}, v'_{d-3}$  in the same way as above to conclude that $|v_{d-1} \cap \alpha'|$ is bounded as $|v_{d-1} \cap \alpha'|$ was, and similarly for $v'_{d-2}$, if $v'_{d-2} \neq v_{d-2}$. We conclude that the new path after surgery is efficient, since the original path $v_d, v_{d-1}, v_{d-2}, v_{d-3}$ is efficient.

{\bf Spiral surgery preserves distance: }  Proving that $d(v',w) = d(v,w)$ requires that we show that there is no efficient path from $v'$ to $w$ with length shorter than $d(v,w)$. For the sake of contradiction, let us suppose that there is such a path, so $d(v',w) = d' < d(v,w)$ and let's examine the first of the intermediate curves in the path. That is, we choose $x \in \mathcal{C}(S)$ such that $d(v', x)=1$ and $d(x, w) =d'-1$. 
Such an $x$ has $d(v,x)>1$, for otherwise there is a path from $v$ to $w$ through $x$ that has length less than $d(v,w)$. 
\begin{figure}[ht]
  \centering
    \includegraphics[height=2in]{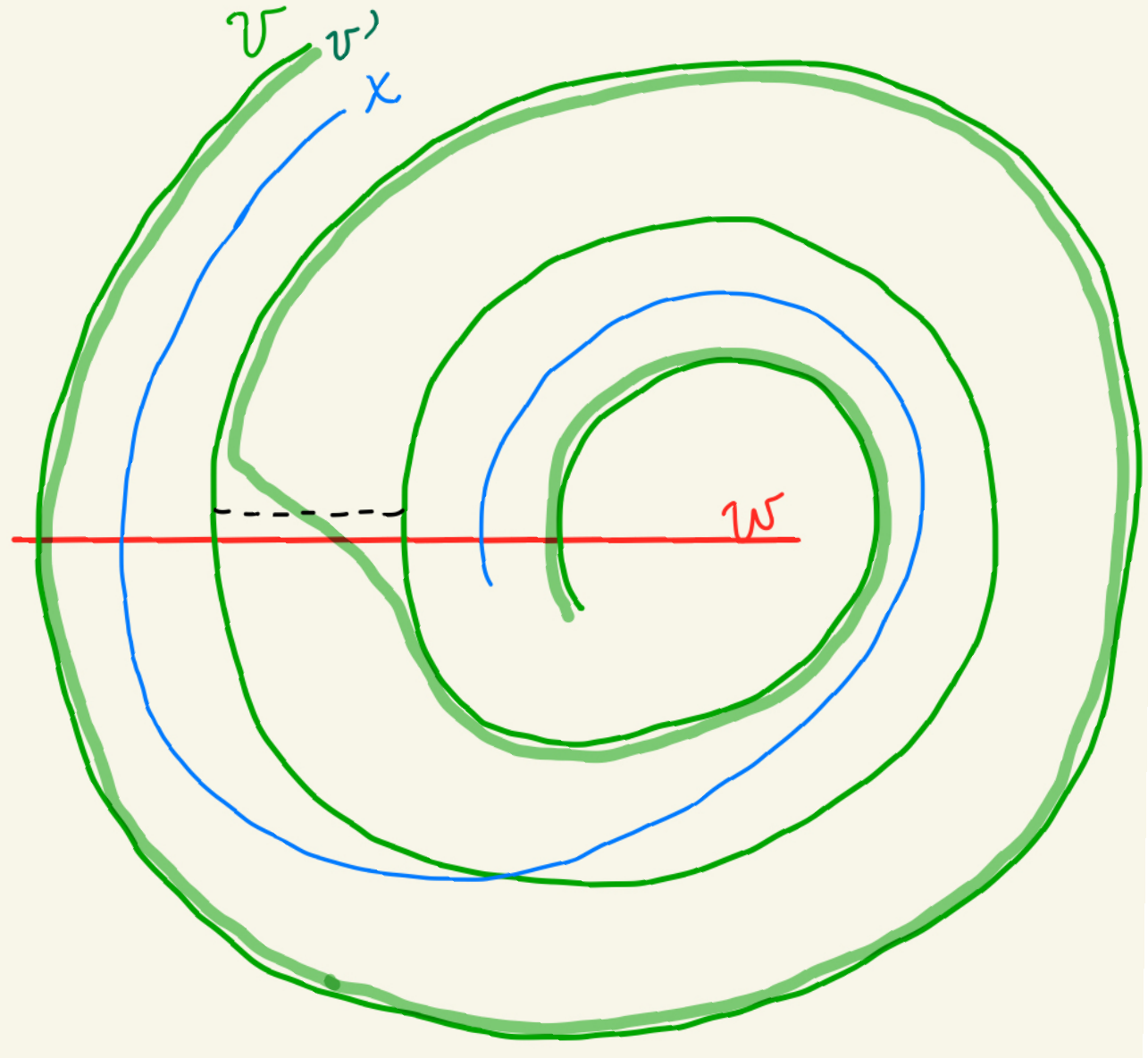}
  \caption{An example to illustrate the proof of Theorem \ref{theorem:main}: the curve $v'$ is the image of $v$ under spiral surgery. The curve $x$ is a curve with $d(x,v) >1$ and $d(x,v') =1$. }
  \label{fig:distance-preserved}
\end{figure}

The assumption that $d(x, v) > 1$ but $d(x,v')=1$ means that $x$ and $v$ only intersected along the interior of the spiral that was deleted by surgery. Since we can assume that the curves are in minimal position, $x$ and $v$ intersect along the spiral only once. See Figure \ref{fig:distance-preserved}. However, this gives us a contradiction thanks to Lemma \ref{lemma:spiral-dot-graph}: as $x$ crosses $v$, the dot graphs across the interior of the spiral would change. This implies that there was no spiral surgery possible. We conclude that there is no such $x$, so $d(v',w) = d(v,w)$. 
  \end{proof}
  
Theorem \ref{theorem:main} tells us that the spiral width and the stacked, extended dot graph above the spiral interior are necessary and sufficient for recognizing when spiral surgery will preserve distance. 
 \section{Conclusion \label{sec:conclusion}}
 
 The basic problem that was studied in this paper is how to reduce intersection number between two curves without reducing distance. The tools that we used were closely related to those in  \cite{BMM}, where the concept of an efficient geodesic joining two vertices in the non-separating curve graph first appeared.  Reference arcs, defined in \cite{BMM}, give a natural metric for the complexity of the geodesics when we consider them in this paper's context: the polygonal decomposition. In particular, certain surgeries that had played a major role in \cite{BMM} led us in a natural way to the discovery of {\it spiral surgery}.  
 
 In Section \ref{ss:extended-dot-graphs}, we formalized how a surgery can reduce the intersection number between two curves while preserving both the length and the efficiency of a given efficient path between them.
 Section \ref{ss:main-theorem} then applied these results to a specific surgery along sequences of rectangles in $\Dec{v,w}$ and proved that such a surgery can maintain distance in $\mathcal{C}(S)$ while reducing the intersection number.

 The proof of Theorem \ref{theorem:main} provides a new lens through which one can view the local geometry of $\mathcal C(S)$.
 Let $(v,w)$ be a filling pair, and let $\mathcal G = \{\mathcal G_1, \ldots, \mathcal G_N\}$ be the efficient geodesics between $v$ and $w$.
 Let $\mathbb{S}(\cdot)$ denote the application of a surgery to a curve or set of curves on $S$. 
 As we've seen here and in \cite{BMM}, the operator $\mathbb{S}(v,w)$ implicitly operates on all paths between $(v,w)$, so for a path $p \subset \mathcal C(S)$, we can write $\mathbb{S}(p) = \{\mathbb{S}(c)| c \in p\}$.

One can also restrict attention to the action of $\mathbb{S}$ on the set of geodesics $\mathcal G$ and, in particular, the subsets of $\mathbb{S}(\mathcal G)$ with length equal to $k$, $\mathbb{S}_k(\mathcal G) = \{ \mathcal G_i \mid \, |\mathbb{S}(\mathcal G_i)| = k\}$.

We consider when $\mathbb{S}$ changes the distance between $(v,w)$:
 $\mathbb{S}(\mathcal G) = \cup_{k=1}^{d(v,w)-1}\mathbb{S}_k$ or $\mathbb{S}(\mathcal G) = \cup_{k=d(v,w)+1}^\infty \mathbb{S}_k$, respectively.
 We now know that, in the case of spiral surgery, the former occurs when the assumptions of Theorem \ref{theorem:main} are violated.
 Regarding the latter, Aougab and Taylor provide an algorithm to construct an infinite geodesic ray, using repeated applications of Dehn twists \cite{AT}. 
 At each step of their algorithm, a geodesic is produced that is strictly longer than the previous one. 
 We conjecture that a surgery-based algorithm of a similar flavor exists, using our machinery of $\Dec{v,w}$ and efficient geodesics. Let's explore this idea, first from the perspective of $\Dec{v,w}$. 

Figure \ref{fig:d3-to-d4}-left is a genus 2 surface cut open along a green curve $v$ that is distance 3 from red curve $w$.  In Figure \ref{fig:d3-to-d4}-right, the same surface is cut open along a modified green curve $v '$ paired with the same red curve $w$.  The difference between the surfaces $S-v$ and $S-v'$ is that there are two additional rectangles in $S-v'$ marked by edges of $w$ between the $6$-gons. The distance between $v'$ and $w$ is 4 (confirmed by MICC). In fact, Figure \ref{fig:d3-to-d4}-right is the same decomposition as that in Figure \ref{fig:d4-polygons-only}; here we have glued the polygons of Figure \ref{fig:d4-polygons-only} together along their $w$-edges to get $S-v'$. We obtained $v'$ from $v$ by performing two operations, each the inverse of a spiral surgery, a transformation we call {\em spiral addition}. 

This spiral addition that increases distance is related to the tools used in \cite{AT}, where the growth of the minimum intersection number between vertices in the curve graph was studied with respect to distance and genus. It is also related to the tools used in \cite{PS} and most recently \cite{Ras}, where the hyperbolicity of the non-separating curve graph was studied with an emphasis on the techniques of combinatorial topology. 

\begin{remark}Comparing Figure \ref{fig:d3-to-d4}-right to  Figure \ref{fig:d3-to-d4}-left, and consulting Figure \ref{fig:curlicue_example_with_4gon}, one sees that two one-rectangle spirals have been surgered in the passage from the right sketch to the left sketch.  The data we obtained from MICC had told us that $i_{min}(4,2,(0,2,0,0)) = 12$, therefore, by Theorem \ref{theorem:main}, the distance would necessarily have gone down to 3 if we had deleted either one-rectangle spiral instead of both. Indeed, computations with MICC revealed even more to us:  if the left spiral in the right sketch had been made arbitrarily long,   and edge numbers increased appropriately, we would still have a pair of filling curves on a surface of genus 2 whose intersection number increased arbitrarily but whose distance remained at 4. This is essentially the construction that was given by Leasure in \cite{Lea}, when he proved that it is possible for intersection number to increase arbitrarily while distance remains fixed.  \end{remark}
\begin{figure}[ht]
  \centering
  \includegraphics[width=2.5in]{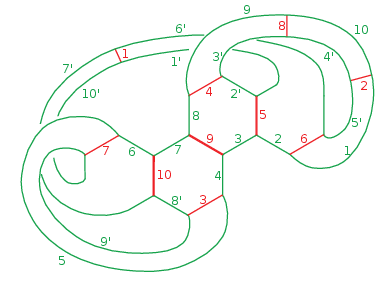}
\hspace{.5cm}
\includegraphics[height=2.5in]{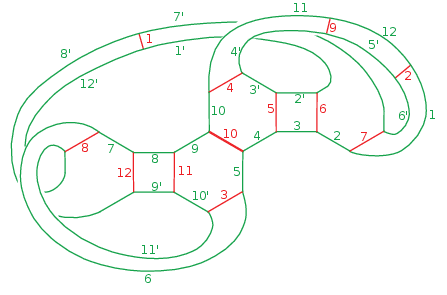}
\caption{
Left: a genus 2 surface cut open along a green curve distance 3 from the red curve. Right: same surface cut open along a curve distance 4 from the same curve. Note this is the polygon decomposition from Figure \ref{fig:d4-polygons-only}. }
  \label{fig:d3-to-d4}
\end{figure}

  \begin{figure}[ht]
  \centering
  \includegraphics[width=2in]{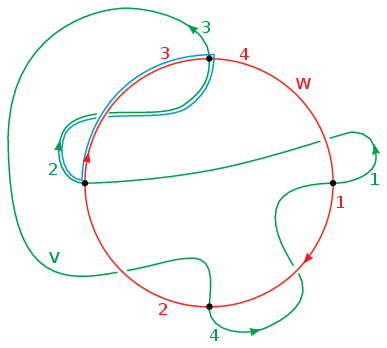} 
  \hspace{.3in}
    \includegraphics[width=1in]{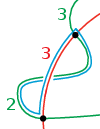}
  \hspace{.4in}
\includegraphics[width=2in]{figures/example_with_spiral.eps}

\caption{The distance-preserving transformations of these curve pairs, from left to right, are given by spiral addition. As noted earlier, the decomposition on the left consists of a single $12$-gon and a single $4$-gon, on the right, a single $12$-gon and a two $4$-gons. In both cases, $d(v,w) = 3$. }
 \label{fig:spiral-addition-as-dehn-twist}
\end{figure}
  \begin{figure}[ht]
  \centering
  \includegraphics[width=1.5in]{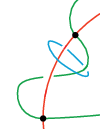} 
  \hspace{.5in}
    \includegraphics[height=1.5in]{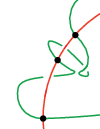}
  \hspace{.2in}

\caption{The realization of spiral addition by a Dehn twist of the green curve $v'$ about the blue curve. }
 \label{fig:spiral-addition-as-dehn-twist-2}
\end{figure}

To explain these connections, let's return to the example in $\S$\ref{subsection:motivating eg}. See Figure~\ref {fig:spiral-addition-as-dehn-twist}, where we have reversed the order of the two sketches in Figure~\ref{fig:curlicue_example}, and added a new sketch in the middle. Recall that filling pairs such as the curves $v$ (in green) and $w$ (in red) intersect many times and so determine in a natural way a collection of {\it bicorn curves}  \cite{PS,Ras}. A bicorn curve is a curve formed from the filling pair $v\cup w$ as the union of two simple arcs, one a subarc of $v'$ and the other a subarc of $w$, where the arcs intersect only at their endpoints.  In the example in Figure~\ref{fig:spiral-addition-as-dehn-twist} the bicorn, shown in red and green,  is formed from arc 2 of $v'$ and arc 3 of $w$ in the left sketch.  Its push-off from $v$ and $w$ into the interior of the positively oriented surface $S$, which we denote $\beta$, intersects $v'$ once, and in the figure is blue.  Figure~\ref{fig:spiral-addition-as-dehn-twist-2}  shows that $v$ is the image of $v'$ under a Dehn twist $T_\beta$ about the push-off $\beta$ of the bicorn. Thus, the addition of an $m$-rectangle  spiral will replace $v'$ with $T_\beta^m(v')$,  relating our work to the methods used in \cite{AT} to extend a given geodesic.
  
The example in Figure \ref{fig:d3-to-d4} above was constructed similarly. In Figure \ref{fig:d3-to-d4}-left, there are two bicorn curves whose push-offs we perform simultaneous Dehn twists around: red 5 $\cup$ green 2, and red 10 $\cup$ green 7. A bicorn curve in the decomposition can be recognized by the same color edge meeting both ends of the other color edge.  Using these bicorns to increase the distance as in these examples only works to take us from distance 3 to distance 4, though. That is, these simple additions won't get the curves any further apart in the curve graph. This is because the bicorn curve in these examples intersects each of $v'$, $v$, and $w$ once, so it is always distance 2 from each. We conjecture that there is a generalization of this spiral addition that leads to a distance-increasing algorithm as well-controlled as spiral surgery is. 
 
 All this is to say that (i) spiral reduction and spiral addition are closely related problems, and (ii) connections exist between our work and the problems solved in \cite{AT} and \cite{PS}, where Dehn twists along generalized bicorns are the primary tools for creating paths between vertices in the curve graph, and extending geodesics to longer ones.  Work is in progress on both of these problems.

{\bf Future directions for student research:}  We have noticed  emerging from this work many questions for an interested student to consider. We haven't studied the decompositions for $d=3$ at all, and it would be an interesting question to explore such decompositions more fully. In a perhaps related direction of possible study for students, the surgeries we analyzed in this paper, which reduce the intersection number by reducing the number of rectangles in $\Dec{v,w}$, could also be applied to $2k$-gons to divide them into smaller $2m$- and $2n$-gons. For example, in the Appendix of this paper, we reduce intersection number to the minimum known for genus $2$ and a decomposition consisting of one $12$-gon. It is possible to perform what one could call an inverse  surgery ``across the $12$-gon" to divide it into two $8$-gons, while preserving distance, and then we were able to further reduce the intersection number to $12$ by surgeries across rectangles. This was only a single example, but it raised some exciting questions. We needed to increase intersection number when we divided the $12$-gon into two $8$-gons, before we we were able to reduce intersection number further. Could this non-monotonic reduction of intersection number been avoided? 

Finally, we point out that spiral surgery is, of course, not the only way that one might reduce intersection number while preserving distance. We explore some non-spiral surgeries in the appendix, Section \ref{sec:Hempel}, which are still $-+$ or $+-$ surgeries across rectangles. Given that we only studied spirals, not $k$-spirals, which are the types of spiraling that appear in the Hempel example, a more formal analysis of $k$-spiral surgery (including a definition!,) would be a valuable project. We also recognize that we did not consider $++$ or $--$ surgeries, which necessarily change the decomposition of a surface (this follows from Euler characteristic arguments like those in Section \ref{ss:decomposition}.) 

There are many exciting directions for this work, and one underlying hope is that by focusing on these low-distance examples, one might gain an intuition for distance 4 that matches the simplicity of distance 3's equivalence to filling. 
  
\section{Appendix: the Hempel example \label{sec:Hempel}}

In Saul Schleimer's 2006 notes on the curve complex \cite{SS}, he describes an example given by John Hempel of a pair of filling curves $(v,w)$ on a genus $2$ surface with $i(v,w)=25$, which were asserted to be distance $4$. Until \cite{BMM} and \cite{MICC}, this was the only explicit example of a purported distance 4 curve pair. Using the efficient geodesic algorithm developed in \cite{BMM}, it was proved that the Hempel example is indeed distance $4$ \cite{GMMM}.  

\begin{figure}[ht]
  \centering
  \includegraphics[height=2in]{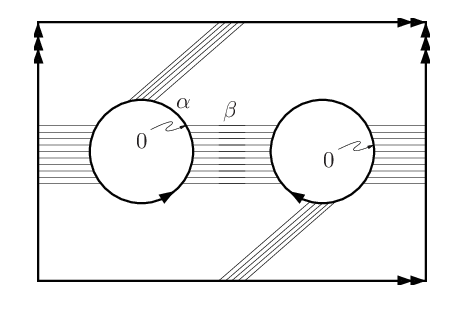}
\caption{The sides of the rectangle are glued as shown. The two circles are glued by a reflection followed by $4/25$ of a rotation to make the marked points agree. The result is a closed genus two surface. The circles become a simple closed curve $\alpha$ and the light lines close up to give $\beta$. We will call $\alpha$ and $\beta$ $v$ and $w$. Figure from \cite{SS}.}
  \label{fig:SS} 
\end{figure}

We perform surgery on the curves in the Hempel example to reduce their intersection number to $16$ while preserving their distance. The surgery we perform in this example is similar to, but not exactly, the spiral surgery of this paper. It is $+-$ or $-+$ surgery across rectangles, guided by the polygonal decomposition as the spiral surgery was. We note that, while we have not formally defined $k$-spiral surgery, some of these surgeries are across $k$-spirals. We used MICC to confirm that each surgery in this example resulted in a curve pair with distance $4$. We draw our surgery arcs slightly differently in this example than we did in Section \ref{subsection:spiral surgery}, because it made the transformation of the curves more manageable and a little easier to draw. Below are the two renditions of surgery arcs across a $4$-gon and the indicated piece of $v$ that would be deleted by surgery across each arc. 

\begin{figure}[ht]
  \centering
  \includegraphics[width=5in]{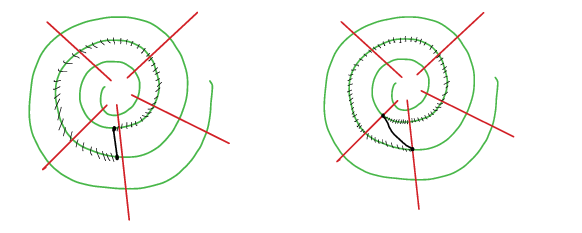}
\caption{On the left is the surgery arc used throughout the paper. On the right is the surgery arc we use in this example.}
  \label{fig:surgery-arcs} 
\end{figure} 

First, we describe the curve pair as given in the example. Cutting $S$ along $v$ and $w$ yields one $12$-gon and $16$ $4$-gons. Gluing the $4$-gons along arcs of $w$, we get three bands of rectangles pairing $w$-edges of the $12$-gon; a grey band of length 4 and two (pink and blue) each of length 8. 

While cyclic embeddings like those of Figure \ref{fig:curlicue_example} are visually appealing, they are impractical for computations relevant for this discussion. 
We need a straightforward representation of curve pairs that allow for easy visual inspection of $\Dec{v,w}$.
Our preferred choice is the {\it Menasco ladder}, first defined in \cite{GMMM}.

\begin{figure}[ht]

  \centering

  \includegraphics[width=6in]{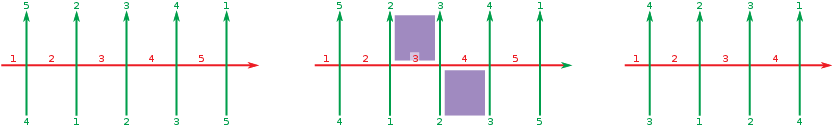}
  \caption{Left and right are the Menasco ladder representations of $(v, w)$ and $(v',w)$ from Figure \ref{fig:curlicue_example}.
  In the middle we have copied the left ladder and shaded the rectangle that was removed from the decomposition $\Dec{v,w}$ when we performed surgery.}
  \label{fig:example-of-ladder}
\end{figure}

To produce the Menasco ladder in Figure \ref{fig:example-of-ladder} from the cyclic representations of the curve pairs in Figure \ref{fig:curlicue_example}, we cut open $w$ on the interior of its arc labeled 1, and we cut open each arc of $v$ between each pair of intersection points. Laying the cut-open $w$ horizontally, such that $w^1, \ldots, w^n$ are listed sequentially, and the cut-open arcs of $v$ vertically, we get a sideways ``ladder", which we use to present the curves in a tabular form. We have two vectors from the ladder, one for the top labels of $v$ and one for the bottom. (The Menasco ladder is also the primary way of inputting the curves into MICC for analysis.) In Figure \ref{fig:example-of-ladder}, we have curves represented by ladder $[5,2,3,4,1],[4,1,2,3,5]$ transformed via surgery into curves represented by ladder $[4,2,3,1], [3,1,2,4]$. 

One can recognize spirals in the Menasco ladder notation in much the same way that one recognizes single rectangles: as subsequences of numbers along the ladder that repeat once on the opposite side, then increase by 1 and repeat again on the other side, with the increase/repetition continuing with each ``spoke" of the spiral.

 We begin with the original ladder notation for the Hempel example from MICC: 

\begin{align*}
(v, w)&=\\
       &[1, \,7, {\bf 18, 24,} \,5, 16, 12, 8, 19,25,6,17,23,4,15,11,22,3,14,10,21,2,13, 9,20]\\
       &[25,6, 17, 23, 4, 15, 11, 7, {\bf 18,24,}\, 5,16,22,3,14,10,21,2,13, 9,20,1,12, 8,19]
\end{align*}

\begin{figure}[ht]
  \centering
  \includegraphics[width=6in]{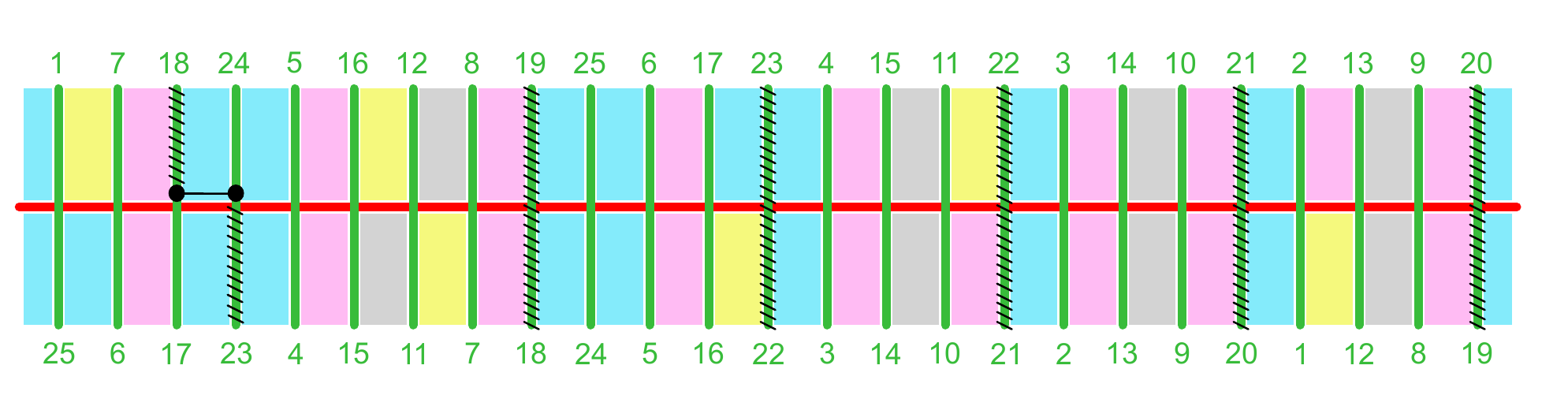}
\caption{The curves from Figure \ref{fig:SS} represented as a Menasco ladder, with surgery arc in black, and hashmarks indicating the piece of $v$ that will be deleted by surgery. The $12$-gon is yellow, and the bands are blue, pink, and grey.}
  \label{fig:i25} 
\end{figure} 

After this surgery, the face decomposition of $S - (v', w)$ still consists of a single $12$-gon and some number of $4$-gons, $i(v', w) = 19$, and the ladder given by: 

\begin{align*}
(v', w)&=\\
       &[7,18,5,16,12,8,19,6,17,4,15,11,3,14,10,2,13,9,1]]\\
       &[6,17,4,15,11,7,18,5,16,3,14,10,2,13,9,1,12,8,19]
\end{align*}

\begin{figure}[ht]
  \centering
  \includegraphics[width=6in]{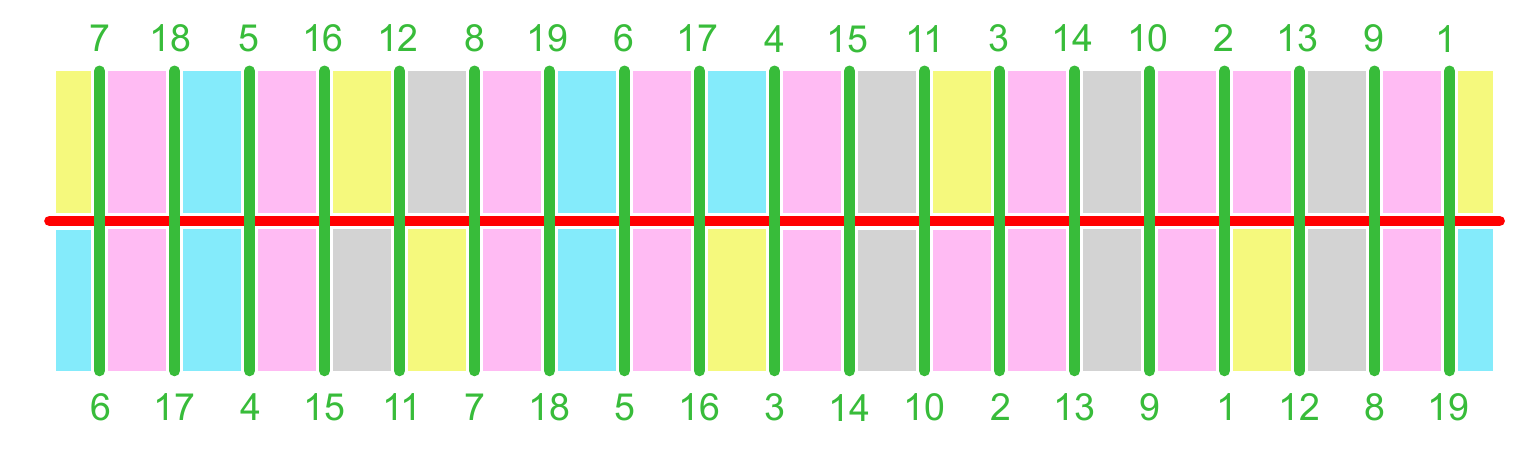}
\caption{The curves after the surgery indicated in Figure \ref{fig:i25}, with relabeling. The $12$-gon is yellow, and the bands are blue, pink, and grey. Any further surgery across the bands will result in reduced distance.}
  \label{fig:i19} 
\end{figure} 

At this point, we have three bands as before: a blue band of length 3, a grey band of length 4, and a pink band of length 8.  Further surgery on $v'$ is impossible:
\begin{enumerate}
\item The yellow band of length 3 is already minimal, we can't reduce the number of $4$-gons further without reducing the distance from $v_1$ to $w$. 
\item Here we rule out the surgery by considering the width of the band. The width of the pink band of length 8 is  $min\{11, 19-11=8\}$, so surgery along that band would result in $i(v'',w) =11$, which is less than the minimum $i$ for distance $4$ on genus $2$. 
\item The grey band of length 4 has width of min$\{4, 19-4=15\}=4$, so surgery along that band would result in a curve $v''$ with $i(v'', w)=15$. However, it is easy to check that any surgery along this band results in the $v'$-edges labeled $11$ and $12$ being deleted. As a result, two edges of the $12$-gon previously sharing the vertex separating $v$-edge $11$ and $12$ now share a single edge, so $d(v'',w)=3$; that is, there is a $v''_1$ distance $2$ from $w$.
\end{enumerate}

We reverse the roles of $v'$ and $w$ by gluing the $2k$-gons of the decomposition along $v$-edges and considering the surface $S-w$. We then look for surgeries to transform $w$ into a curve $w'$ that is the same distance from $v'$ but with smaller $i(v', w')$. 
We have the following ladder notation for $(w, v')$:

\begin{align*}
(w,v')&=\\
       &[5, 17, 14, 11, 8, 1, 6, 18, 4, 16, {\bf 13, 10}, 3, 15, 12, 9, 2, 7, 19]\\
       &[6, 18, 15, 12, 9, 2, 7, 19, 5, 17, 14, 11, 4, 16, {\bf 13, 10}, 3, 8, 1]
\end{align*}

The face decomposition of $S-(w, v')$ is, of course, the same as that of $S-(v', w)$.  We could have glued the $4$-gons along their $v'$-edges without reordering the ladder, but it is easier to analyze the bands and possible surgeries on their $w$-edges with the ladder notation showing the long $w$-edges. We now have one long orange band whose long sides are identified with the long sides of a shorter dark grey band, and one long purple band that almost forms a spiral (long sides identified with themselves but only on one side.) We will perform the surgery on the purple band. We note that the purple band has width $3$, so after surgery the intersection number will remain above the minimum for this distance and genus. 

\begin{figure}[ht]
  \centering
  \includegraphics[width=6in]{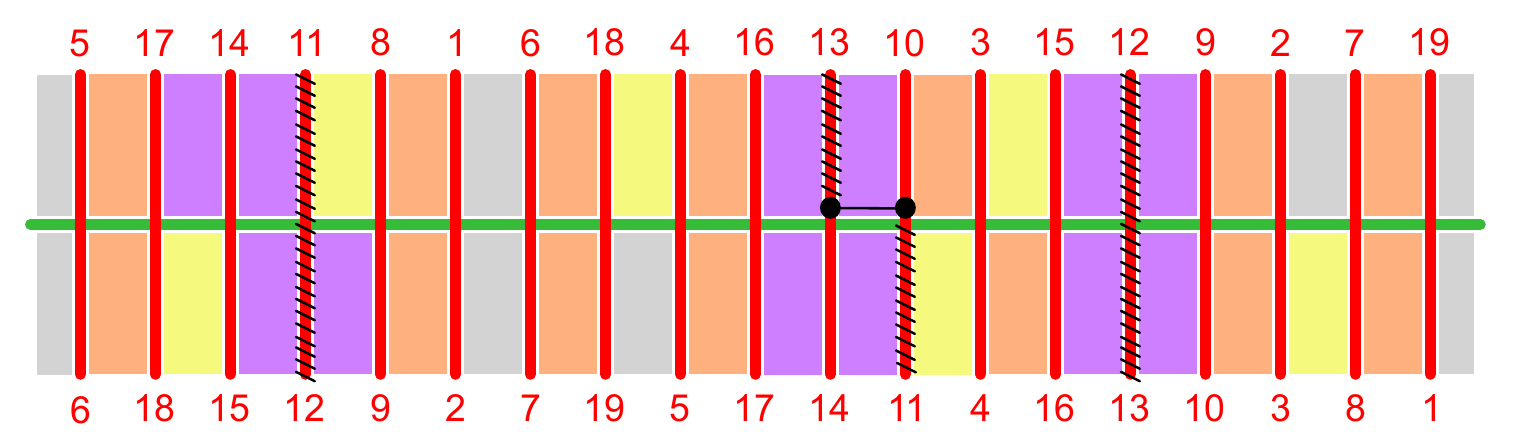}
\caption{The curves from Figure \ref{fig:i19} with $v'$ and $w$ reversed in the ladder representation. The $12$-gon is yellow, and the bands are purple, orange, and dark grey. A surgery arc is indicated in black, and the piece of $w$ that will be deleted by surgery is hashed out in black.}
  \label{fig:i19-reversed} 
\end{figure}

The ladder notation for $(w', v')$ after this surgery on $w$ is given below, and $d(w',v')=4$. 

\begin{align*}
(w',v')&=\\
       &[5, 14, 11, 8, 1, 6, 15, 4, 13, 10, 3, 12, 9, 2, 7, 16]\\
       &[6, 15, 12, 9, 2, 7, 16, 5, 14, 11, 4, 13, 10, 3, 8, 1]
\end{align*}

\begin{figure}[ht]
  \centering
  \includegraphics[width=6in]{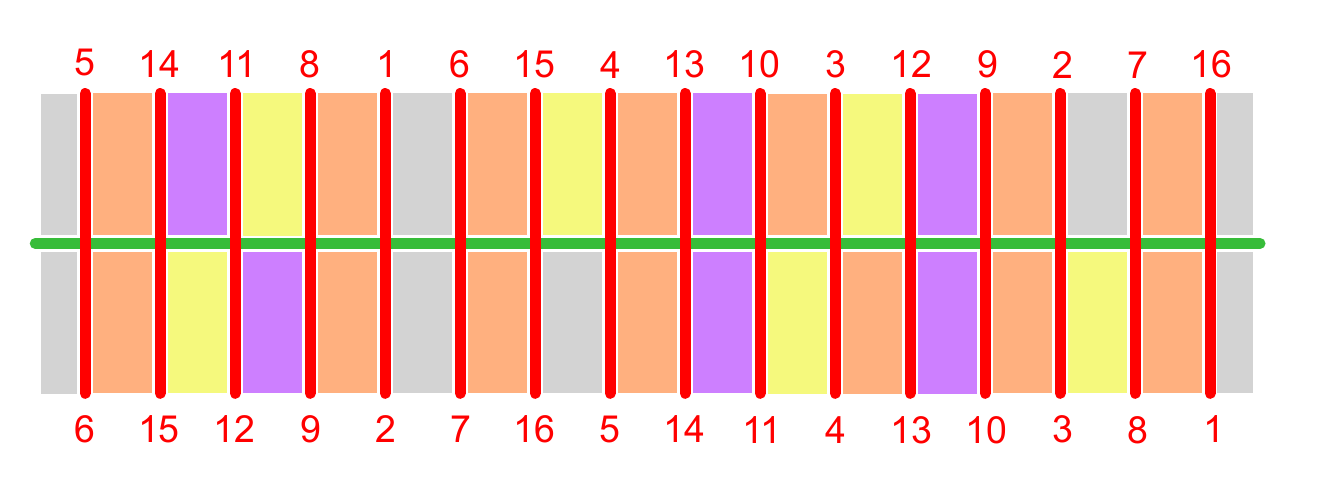}
\caption{The curves $w',v'$ with minimum intersection number ($16$) for the distance ($4$,) genus ($2$,) and decomposition (one $12$-gon.) The $12$-gon is yellow, and the bands are purple, orange, and dark grey.}
  \label{fig:i16} 
\end{figure}

We observe that there is no further surgery across rectangles available to reduce the intersection number of $v'$ and $w'$ further without also reducing the distance between the curves: for consider the three bands in the decomposition $S-\{w', v'\}$, as pictured in the figure above. The orange band has width $7$, and the dark grey band has width $5$, which means surgery on either band will take the intersection number too low for distance $4$ curves on genus $2$. Let $\mathcal{B}$ be the purple band. Surgery on $\mathcal{B}$ is only possible if it doesn't result in the deletion of $w'$-edges $11$ and $12$. Otherwise, a surgery along $\mathcal{B}$ would result in the $v'$ edges of the $12$-gon adjacent to the vertex between $w'$-edges $11$ and $12$ being identified. The $w'$-edges of $\mathcal{B}$ are $12-13-14$ and $9-10-11$; no surgery arc from one side of the band to the other does not include edges $11$ and $12$ (the only possible short surgery arcs are $9$ to $13$ and $10$ to $14$).  Now, we have only demonstrated that this particular sequence of surgeries has now reached a dead end.  However, since MICC have us the complete dataset of all isotopy classes of pairs of curves of distance $\geq 4$ with $i(v,w) \leq 25$ and genus $g = 2$, we know that the minimum intersection number for a curve pair of distance $4$, with a decomposition consisting of one $12$-gon and some number of $4$-gons, is indeed 16.\\

This example shows how critical the use of MICC was in our work. It also shows that, while we were able to prove that spiral surgery preserves distance, it is clearly not the only surgery, even across rectangles, that will preserve distance. The work continues.

\

\noindent Joan S. Birman,  Department of Mathematics, Columbia University, New York, NY 10027 \\
$<$jb@math.columbia.edu$>$

\

\noindent Matthew J. Morse, Department of Computer Science, Courant Institute for Mathematical Sciences, New York University, New York, NY 10012 $<$mmorse@cs.nyu.edu$>$

\

\noindent Nancy C. Wrinkle,  Department of Mathematics, Northeastern Illinois University, Chicago, IL 60640 $<$n-wrinkle@neiu.edu$>$

\

\end{document}